\documentclass[a4paper,11pt]{amsart}
\usepackage{amsmath,amsthm,amssymb,amsfonts,enumerate,color,enumerate}
\usepackage[pdftex]{graphicx}

\usepackage{verbatim}
\usepackage[spanish,USenglish]{babel} 
\usepackage{color}
\usepackage{tikz}
\usepackage{graphicx}

\newtheorem{theorem}{Theorem}
\newtheorem{corollary}[theorem]{Corollary}
\newtheorem{lemma}[theorem]{Lemma}

\newtheorem{proposition}[theorem]{Proposition}
\newtheorem{definition}[theorem]{Definition}

\newtheorem{remark}[theorem]{Remark}

\newcommand{\edproof}{ $\hfill {\Box}$}

\allowdisplaybreaks

\newcommand{\INT}{\int_0^\infty}

\newcommand{\F}{\mathcal{{F}}}
\newcommand{\R}{\mathbb{{R}}}

\newcommand{\N}{\mathbb{{N}}}

\newcommand{\C}{\mathbb{{C}}}
\newcommand{\LL}{\mathcal{{L}}}
\newcommand{\HH}{\mathcal{{H}}}

\author[De Le\'on-Contreras]{Marta De Le\'on-Contreras}
\address{Departamento de Matem\'aticas, Facultad de Ciencias, Universidad Aut\'onoma de Madrid, 28049 Madrid, Spain.}
\email{marta.leon@uam.es}

\author[Torrea]{Jos\'e L. Torrea}
\address{Departamento de Matem\'aticas, Facultad de Ciencias, Universidad Aut\'onoma de Madrid, 28049 Madrid, Spain.}
\email{joseluis.torrea@uam.es}

\thanks{Research partially supported by grant  MTM2015-66157-C2-1-P (MINECO/FEDER)}

\keywords{Semigroups. Fractional laplacian. Lipschitz H\"older Zygmund spaces. H\"older estimates}

\subjclass[2010]{Primary 42C05; Secondary 35K08, 42B35}

\begin{document}

\title{Fractional powers of the  parabolic Hermite operator. Regularity properties}

\begin{abstract}
Let $\LL= \partial_t- \Delta_x+|x|^2$. Consider its Poisson semigroup $e^{-y\sqrt{\LL}}$. For $\alpha >0$ define the 
Parabolic Hermite-Zygmund spaces 
		$$
	\Lambda^\alpha_{\mathcal{L}}=\left\{f: \:f\in L^\infty(\mathbb{R}^{n+1})\:\; {\rm and} \:\;  \left\|\partial_y^k	e^{-y\sqrt{\LL}} f \right\|_{L^\infty(\mathbb{R}^{n+1})}\leq C_k y^{-k+\alpha},\;\: {\rm with }\, k=[\alpha]+1, y>0. \right\},
	$$
	with the obvious norm. It is shown that these spaces have a pointwise description of H\"older type. 
	 
	The fractional powers $\LL^{\pm \beta}$ are well defined in these spaces and the following regularity properties are proved:
	\begin{eqnarray*} 
	\alpha, \beta >0, \quad \|\mathcal{L}^{-\beta} f\|_{	\Lambda^{\alpha+2\beta}_{\mathcal{L}}}\le C \|f\|_{	\Lambda^\alpha_{\mathcal{L}}}.
	\end{eqnarray*}
	\begin{eqnarray*}
	0< 2\beta < \alpha, \quad  \|\mathcal{L}^\beta f\|_{\Lambda_{\mathcal{L}}^{\alpha-2\beta}}\le C \|f\|_{\Lambda^\alpha_{\mathcal{L}}}.
	\end{eqnarray*}
	Parallel results are obtained   for the Hermite operator $- \Delta +|x|^2.$ The proofs use in a fundamental way the semigroup definition of the operators  $\LL^{\pm \beta}$ and $(-\Delta+|x|^2)^{\pm \beta}$.  The non-convolution structure of the operators produce an extra difficulty of the arguments.  
	
\end{abstract}

\maketitle

\section{Introduction}

Treatises  dealing with  Lipschitz and H\"older spaces have been the object in  quite a lot  papers and books along the last hundred years. In general they can be considered as  
the classes between the space of continuos functions and the space of $\mathcal{C}^1$ (differentiable with  continuous derivatives)  functions, this is the case of  $C^\alpha, \, 0< \alpha <1$. Also they can be considered as the spaces which fill the interval between the classes $\mathcal{C}^k$ and $\mathcal{C}^{k+1}$, this is the case of the spaces $C^{k,\alpha}, k \in \mathbb{N}, 0<\alpha<1$. The importance of the smoothness of the functions in the classical theory of Fourier series drove, in a natural way, to analyze the validity of different theorems for the case of Lipschiz functions.  We refer to the classical book of Zygmund, \cite{Zygmund}, to see the role played by these classes in classical Fourier Analysis.  In Harmonic Analysis the classes became important as spaces in which some operators are well defined and satisfy some boundedness properties,  we refer  to the book of E. Stein, \cite{Stein}, in order to have a detailed description from a Harmonic Analysis point of view.  In differential equations, Lipschitz continuity is the key of the Picard-Lindel\"of theorem for  the existence and uniqueness of the solution to an initial value problem.  Results about regularity properties with respect to H\"older classes, ${C}^{\alpha}(\mathbb{R}^n)$ and ${C}^{k, \alpha}(\mathbb{R}^n)$ ,  are one of the important matters  in the theory of partial differential equations. For elliptic operators they can be used to  obtain classical solutions of second order elliptic equations of the form $Lu=f$ (see for instance \cite[Chapter~6]{Gilbarg-Trudinger}).  Moreover,   in  certain measure spaces without notion of  derivative, the Lipschitz classes are a good substitute of the space $\mathcal{C}^\infty$ in order to define distributions, and some abstract Harmonic Analysis can be performed.  This is of special importance in spaces of homogeneous type, see  \cite{MaciasSegovia}.  Finally they are object of study in their own by researchers in Functional Analysis, see \cite{Kalton}.

The outbreak produced by the paper of L. Caffarelli and L.Silvestre  about the fractional laplacian, \cite{CaffaSilve}, has given way to a flowering of papers analyzing the  classical properties of the elliptic operators but in the case of these  ``new'' fractional operators.  In particular regularity properties for the operator $(-\Delta)^\sigma$ were  proved in  \cite{CaffarelliSilvestre-Reg}.  For elliptic operators in divergence form see \cite{CaffarelliPablo}. In the case of the Harmonic oscillator $\mathcal{H}=-\Delta+|x|^2,$ the classes ${C}^{\alpha}_\mathcal{H}(\mathbb{R}^n)$ were defined in \cite{ST2}, see Definition \ref{DefPablo},  Schauder and H\"older estimates were proved in this case.

As a shorthand it can be said that, for $0<\alpha<1$,  a $C^\alpha$ function satisfies an inequality of the type $|f(x)-f(x-y)| \le C |y|^\alpha$. For $\alpha >1$, not an integer, the $[\alpha]$-order derivatives of the function $f$ satisfy the same kind of inequality. Special mention should deserve the case $\alpha=1$, with is described  as the  {\it Zygmund
class}  $|f(x+y)+f(x-y)-2f(x)| < c |y|$, see \cite[Chapter II]{Zygmund} . See also the interesting article \cite{Krantz1} and the references there in.
These pointwise definitions imply that to prove   regularity results of an operator among these spaces  we need its pointwise expression.  In some (in fact many) cases this can be a rather involved formula, see for example the expressions of $(-\Delta)^\alpha$ and $\HH^{-\sigma} f(x)$ in \cite{ST2}.

In  the 60's of last century  the language of the semigroups was used in order to characterize H\"older spaces, see \cite{Taibleson}. This is specially successful in the case of the Poisson semigroup. The classical reference is  E. M. Stein,  see \cite[Chapter~5]{Stein}. Being a little bit imprecise it can be said  that a function $f$ belongs to a class ${\Lambda}^\alpha$  if $\| \partial^k_t e^{-t\sqrt{-\Delta}}f\|_{L^\infty(\mathbb{R}^n)}\le C t^{-k+\alpha}, k \ge \alpha.$ A posteriori these classes are seen to coincide with the $C^\alpha$ classes. It is interesting to notice that this description also covers the Zygmund class. In the present paper the importance of this picture is based on the fact that in order to prove boundedness properties of operators, one could avoid the long, tedious and sometimes cumbersome computations that are needed when the pointwise expressions are handled. This will be our case.

The characterization of H\"older spaces via the Poisson semigroup $e^{-t\sqrt{-\Delta}}$  raise the question of analyze some H\"older spaces associated to different   laplacians  and to find the pointwise and semigroup estimate characterizations. For the case of the Ornstein-Ulhenbeck operator in $\mathbb{R}^n$, $O=\frac1{2}\Delta -x\cdot \nabla$, the so-called {\it Gaussian Lipschitz spaces} were defined in 
\cite{GattoUrbina} as the collection of functions such that 
$\|\partial_y^k	e^{-y\sqrt{O}} f \|_{L^\infty(\mathbb{R}^n,\frac{e^{-|x|^2}}{\pi^{n/2}})}\leq C_k y^{-k+\alpha}, k = [\alpha]+1,$ where  $e^{-y\sqrt{O}}$ is the Poisson semigroup associated to the operator $O$. In the particular interval $0 < \alpha <1$, these Gaussian Lipschitz spaces have been recently characterized 
pointwise in \cite{Sjogren}.  If  $\mathcal{S} = -\Delta+ V$  is the Schr\"odinger operator in $\mathbb{R}^n, n \ge 3$,
 where $V$ satisfies satisfies a reverse H\"older inequality for some $q> n/2$, the classes $\Lambda_\mathcal{S}^\alpha, 0<\alpha <1$, were defined in \cite{TaoStingaTorreaZhang}.  The
authors prove that the classes can be described by a Campanato-BMO type condition, boundedness in these spaces of operators like fractional powers of $\mathcal{S}$  are considered. 
For the Hermite operator $\HH= -\Delta + |x|^2$ in $\mathbb{R}^n$, pointwise H\"older spaces, $C^{k,\alpha}_\HH$  were defined in \cite{ST2} and boundedness properties of Hermite fractional laplacian, $\HH^\alpha, 0 < \alpha < 1$, were considered.  In the case of parabolic operators of the type  $ \frac{\partial}{\partial t} u(t,x)= a^{ij}(t,x) \frac{\partial^2}{\partial x_i \partial x_j} u(t,x)+b^i(t,x) \frac{\partial}{\partial x_i }u(t,x)+c(t,x)u(t,x)+f(t,x),$ where $a,b,c $ are real valued and $c\le 0$, some pointwise  H\"older classes were introduced in \cite{Krylovbook}. Where    solvability  and a priori estimates were proved. For $\mathcal{M} =\partial_t + \Delta$, the  Poisson semigroup $e^{-y\sqrt{\mathcal{M}}}$  is used in \cite{ST3} for defined the corresponding H\"older classes. 
 The coincidence with the pointwise classes of Krylov were proved for the   $\alpha$ considered in \cite{Krylovbook}. This semigroup characterization was used to show new  regularity properties for fractional powers  $(\partial_t+\Delta_x)^{\pm \alpha}$.

 Now we shall present our results.

Along this paper we shall deal with the {\it parabolic Hermite operator}
\begin{equation}\label{oscilador}  \mathcal{L} := \partial_t + \HH= \partial_t -\Delta_x+ |x|^2, \, x \in \R^n, \, t >0.
\end{equation}
 As the operators $\partial_t$ and $\HH$ commute, the  heat semigroup   $e^{-y \LL}$  will be the composition 
of the heat semigroups $e^{-y\partial_t}$ and $ e^{-y \mathcal{H}}.$  As these semigroups are  well known, see \cite{Stempak} and \cite{Bernardis}, we shall have a satisfactory description of the operator  $e^{-y \LL}$. This description will be use  in order to define, among other operators, the Poisson semigroup $e^{-y\sqrt{\LL}} $, the {\it fractional  parabolic Hermite integrals} $\LL^{-\beta}, \, \beta >0$ and the {\it
fractional parabolic Hermite laplacian} $\LL^\beta, \, \beta >0$. See Section \ref{preliminares}.

Once the Poisson semigroup, $\mathcal{P}_y$   is introduced, see Section \ref{preliminares},  we  define the following associated classes of functions.
\begin{definition}\label{ZygmundPar}[Parabolic Hermite-Zygmund spaces] Let $\mathcal{P}_y = e^{-y\sqrt{\LL}}$ and  $\alpha>0$, we consider the class
		$$
	\Lambda^\alpha_{\mathcal{L}}=\left\{f: \:f\in L^\infty(\mathbb{R}^{n+1})\:\; {\rm and} \:\;  \left\|\partial_y^k	\mathcal{P}_y f \right\|_{L^\infty(\mathbb{R}^{n+1})}\leq C_k y^{-k+\alpha},\;\: {\rm with }\, k=[\alpha]+1, y>0. \right\},
	$$
	whose norm is given by $\|f\|_{	\Lambda^\alpha_{\mathcal{L}}}:=\|f\|_\infty+ C$, where $C$ is the infimum of the positive constants $C_k$ above.
\end{definition}
We will show, in Theorem \ref{teo151} that these classes have a pointwise description. Moreover, a restriction to functions depending only on $x$, produces a natural Definition  \ref{ZygmundPar2} and a Theorem \ref{teo16} for the case of Hermite operator in $\mathbb{R}^n$. 

 The operator $\HH$ can be factorized as $\HH= \frac12\sum_{i=1}^n (A_i A_{-i}+A_{-i}A_i),\, A_i=\partial_{x_i}+x_i, \,  A_{-i}=-\partial_{x_i}+x_i$. The first order operators $A_{\pm i}$ play the role, with respect to operator $\HH$,  of the derivatives $\pm \partial_{x_ i} $ with respect to the classical laplacian $\Delta$. See \cite{Stempak}, \cite{ST2}.

\begin{theorem}\label{teo151} Let $\mathcal{L} :=  \partial_t + \HH= \partial_t -\Delta_x+ |x|^2, \, x \in \R^n, \, t >0.$

	\begin{enumerate}
		\item Suppose that $0<\alpha<2$. Then $f\in 	\Lambda^\alpha_{\mathcal{L}}$ if and only if there exists a constant $C>0$ such that
		\begin{equation}\label{eq21} 
		\|f(\cdot-\tau,\cdot-z)+f(\cdot-\tau,\cdot+z)-2f(\cdot,\cdot)\|_{L^\infty(\mathbb{R}^{n+1})} \leq C(|\tau|^{1/2}+|z|)^{\alpha}, (\tau,z) \in\mathbb{R}^{n+1} 
		\end{equation}
	and $(1+|x|)^\alpha f\in L^\infty(\mathbb{R}^{n+1})$. In this case, if $K$ denotes the least constant $C$ for which the inequality above is true, then $\|u\|_{	\Lambda^\alpha_{\mathcal{L}}}:=[u]_{M^\alpha}+ K$.
		Where $[f]_{M^\alpha}=\|(1+|\cdot|)^\alpha f(\cdot,\cdot)\|_\infty.$

		\item Suppose that $\alpha>2$. Then $f\in	\Lambda^\alpha_{\mathcal{L}}$ if and only if
		$$
		A_i A_j f\in 	\Lambda^{\alpha-2}_{\mathcal{L}}, \:\;i,j=\pm1,\dots,\pm n,
		\quad \hbox{  and  } \quad 
		{\partial_t}f \in 	\Lambda^{\alpha-2}_{\mathcal{L}}.
		$$
		In this case the following equivalence holds
		$$
		\|f\|_{\Lambda^{\alpha}_{\mathcal{H}}}\sim  \sum_{i,j=\pm1}^{\pm n}\left(\|A_i A_j f\|_{	\Lambda^{\alpha-2}_{\mathcal{L}}}\right)+ \|{\partial_t}f \|_{	\Lambda^{\alpha-2}_{\mathcal{L}}}.
		$$

\end{enumerate}
\end{theorem} 

The above results have the following parallel results in the case of Hermite operator	 $\HH= -\Delta_x+ |x|^2.$
\begin{definition}\label{ZygmundPar2}[Hermite-Zygmund spaces] Let $P_y = e^{-y\sqrt{\HH}}$ and  $\alpha>0$, we consider the class
		$$
	\Lambda^\alpha_{\mathcal{H}}=\left\{g: \:g\in L^\infty(\mathbb{R}^{ng})\:\; {\rm and} \:\;  \left\|\partial_y^k	P_y g \right\|_{L^\infty(\mathbb{R}^{n})}\leq C_k y^{-k+\alpha},\;\: {\rm with }\, k=[\alpha]+1, y>0.  \right\},
	$$
	whose norm is given by $\|g\|_{	\Lambda^\alpha_{\mathcal{H}}}:=\|g\|_\infty+ C$, where $C$ is the infimum of the positive constants $C_k$ above.
\end{definition}

\begin{theorem}\label{teo16}  Let $g\in L^\infty(\mathbb{R}^n).$
\begin{enumerate}
		\item Suppose that $0<\alpha<2$. Then $g\in \Lambda^\alpha_{\mathcal{H}}$ if and only if   $(1+|\cdot|)^\alpha g\in L^\infty(\mathbb{R}^{n})$ and there exists a constant $C>0$ such that
		\begin{equation*}
		\|g(\cdot-z)+g(\cdot+z)-2g(\cdot)\|_{L^\infty(\mathbb{R}^n)}\leq C|z|^{\alpha}, \, z \in \mathbb{R}^n.
		\end{equation*}
	 In this case, if $K$ denotes the least constant $C$ for which the inequality above is true, then $\|g\|_{\Lambda^\alpha_{\mathcal{H}}}:=[g]_{M^\alpha}+ K$. Where $[g]_{M^\alpha}=\|(1+|\cdot|)^\alpha g(\cdot)\|_\infty.$ 
		\item Suppose that $\alpha>1$. Then $g\in\Lambda^{\alpha}_{\mathcal{H}}$ if and only if
		$$
		\frac{\partial}{\partial x_i}g\in \Lambda^{\alpha-1}_{\mathcal{H}} \hbox { and }  x_i g \in \Lambda^{\alpha-1}_{\mathcal{H}} \:\;i=1,\dots,n.
		$$
		In this case the following equivalence holds
		$$
		\|g\|_{\Lambda^{\alpha}_{\mathcal{H}}}\sim \|g\|_\infty+ \sum_{i=1}^n \Big( \, \left\|\frac{\partial}{\partial x_i}  g\right\|_{\Lambda^{\alpha-1}_{\mathcal{H}}} + \Big\|x_i g \Big\|_{\Lambda^{\alpha-1}_{\mathcal{H}}} \,  \Big).
		$$
	\end{enumerate}
\end{theorem}

As we said before we shall obtain  regularity results of operators associated to $\LL$ when acting over the classes defined above.  We shall consider positive, negative and imaginary powers of the operators $\LL$ and $\HH$, as well as Riesz transforms. For the appropriated definitions see Section \ref{preliminares}.

\begin{theorem}\label{teoHolder}
	Let $0< 2\beta < \alpha $ and  $f\in\Lambda^\alpha_{\mathcal{L}},$ (respectively  $g\in\Lambda^\alpha_{\mathcal{H}}$), then 
	$\mathcal{L}^\beta f\in\Lambda_{\mathcal{L}}^{\alpha-2\beta}$ (respectively 
	$\mathcal{H}^\beta g\in\Lambda_{\mathcal{H}}^{\alpha-2\beta}$) and
$$	\|\mathcal{L}^\beta f\|_{\Lambda_{\mathcal{L}}^{\alpha-2\beta}}\le C \|f\|_{\Lambda^\alpha_{\mathcal{L}}}, \quad 
	\hbox{(respectively  }	\|\mathcal{H}^\beta g\|_{\Lambda_{\mathcal{H}}^{\alpha-2\beta}}\le C \|g\|_{\Lambda^\alpha_{\mathcal{H}}} \hbox{)}.
	$$
\end{theorem}

\begin{theorem}\label{teoSchau} Let $0< \alpha, \beta.$
		\begin{itemize}
		\item[(i)] Given  $f\in 	\Lambda^\alpha_{\mathcal{L}}$ (respectively
		$g\in 	\Lambda^\alpha_{\mathcal{H}}$), then $\mathcal{L}^{-\beta} f \in 	\Lambda^{\alpha+2\beta}_{\mathcal{L}} $\,  (respectively $\mathcal{H}^{-\beta} g \in 	\Lambda^{\alpha+2\beta}_{\mathcal{H}} $) and
	$$
		\|\mathcal{L}^{-\beta} f\|_{	\Lambda^{\alpha+2\beta}_{\mathcal{L}}}\le C \|f\|_{	\Lambda^\alpha_{\mathcal{L}}}, \, \hbox{(respectively  }
		\|\mathcal{H}^{-\beta} g\|_{	\Lambda^{\alpha+2\beta}_{\mathcal{H}}}\le C \|g\|_{	\Lambda^\alpha_{\mathcal{H}}}  \hbox{)}.
		$$
		\item[(ii)] If $f\in L^\infty(\R^{n+1})$, \,  (respectively $g\in L^\infty(\R^n)$), then 		$$
		\|\mathcal{L}^{-\beta} f\|_{	\Lambda^\beta_{\mathcal{L}}}\le C \|f\|_{\infty}, \,  \hbox{(respectively   }  \|\mathcal{H}^{-\beta} g\|_{	\Lambda^\beta_{\mathcal{H}}}\le C \|g\|_{\infty}  \hbox{)}.
		$$
	\end{itemize} 
\end{theorem}

We  also get the boundedness of the multiplier operator of the Laplace transform type on the spaces $	\Lambda^\alpha_{\mathcal{L}}$ and $	\Lambda^\alpha_{\HH}$. We recall to the reader that the imaginary powers $\lambda^{i\gamma}$ are examples of multipliers of Laplace transform type.
In \cite{TaoStingaTorreaZhang}, this result is proved for every Schr\"odinger operator  when $0<\alpha<1$.  
 \begin{theorem}\label{multiplicador}
	Let $a$ be a bounded function on $[0,\infty)$ and consider
	$$
	m(\lambda)=\lambda^{1/2}\int_0^\infty e^{-s\lambda^{1/2}}a(s)ds, \:\; \lambda>0.
	$$
	Then, for every $\alpha>0$, the multiplier operator of the Laplace transform type $m(\LL)$ (respectively $m(\HH)$) is bounded from $\Lambda_{\mathcal{L}}^\alpha$ (respectively $\Lambda_{\mathcal{H}}^\alpha$)  into itself.
	\end{theorem}
In \cite{TaoStingaTorreaZhang}, this result is proved for every Schr\"odinger operator  when $0<\alpha<1$.

\begin{theorem}\label{Rieszm}
Consider  the Parabolic Hermite Riesz  transforms of order $m\ge 1$ defined by
	$$R_\nu=(A_{\pm 1}^{\nu_1}A_{\pm 2}^{\nu_2}\dots A_{\pm n}^{\nu_n})\LL^{-m/2}\,  \hbox{  and } \,  R_m = \partial_t^m \LL^{-m} $$ where $\nu_i \ge 0, i=1,\dots, n $ and $|\nu|=\nu_1+\dots+\nu_n=m$. Let $\alpha >0$, then   $R_\nu$ and $R_m$ are bounded from $	\Lambda^\alpha_{\mathcal{L}}$ into itself.
	A parallel result holds for the operators $(A_{\pm 1}^{\nu_1}A_{\pm 2}^{\nu_2}\dots A_{\pm n}^{\nu_n})\HH^{-m/2}$ when acting on the spaces $	\Lambda^\alpha_{\mathcal{H}}.$
\end{theorem}

See  \cite{Stempak}, \cite{ST2} and  \cite{thangavelu2} and the references there in for more information about  the hermitian Riesz transforms $A_j \HH^{-1/2}.$

Apart from the above regularity results, our semigroup language allows us to get some maximum principle.

\begin{theorem}\label{Maximump}[Maximum principle] Let $0<\beta <1,\, \alpha > 2\beta$ and  $f\in \Lambda^{\alpha/2,\alpha}_{t,\HH_x}$. Suppose that 
\begin{enumerate}
\item $f(t_0,x_0) = 0$ for some $(t_0,x_0) \in \R^{n+1}$, and
\item $f(t,x) \ge 0$ for $t\le t_0,\, x\in \R^n.$
\end{enumerate}
Then   $\LL^\beta f(t_0,x_0) \le 0$.

Moreover, $\LL^\beta f(t_0,x_0)=0$ if and only if $f(t,x) = 0$ for $t\le t_0$ and $x\in \R^{n}$.
\end{theorem}

The organization of the paper is as follows. In Section \ref{preliminares} we present the mains objects like Poisson semigroup and  fractional powers of operators. We observe that as the operator $\LL$ is not positive, the standard definitions have to be adapted to this complex case. In Section \ref{proofs} we show the coincidence of the spaces $\Lambda_\LL^\alpha$ and $\Lambda_\HH^\alpha$ with some H\"older pointwise spaces defined previously in  \cite{Krylovbook} and \cite{ST3} in the parabolic and Hermite settings. Section \ref{Main}
is devoted to the proof of Theorems \ref{teo151} and \ref{teo16}. Theorems \ref{teoHolder}, \ref{teoSchau}, 
\ref{multiplicador}, \ref{Rieszm} and \ref{Maximump} are proved in Sections \ref{fractional} and \ref{maximum}.
Finally in Section \ref{cuentas} we collect some inequalities needed along the paper.  The non-convolution structure of our operators,  produces  non trivial difficulties and technical computations that we have to solve in each case. This is common to the parabolic case $\LL$ and the Hermite case $\HH$. We present the computations and the results in such a way that the parabolic case includes as particular case the Hermite case. This will be clarified  in the subsections called {\it Elliptic Hermite setting} included at the end the corresponding Sections.

Along this paper, we will use the variable constant convention, in which $C$ denotes a constant that may not be the same in each appearance. The constant will be written with subindexes if we need to emphasize the dependence on some parameters.

\section{Preliminary considerations.}\label{preliminares}

For functions $g\in L^p(\R^n)$, the heat semigroup $e^{-\tau\HH}$ has the pointwise expression
\begin{align*}
e^{-\tau \mathcal{H}}g(x)&=
 \int_{\mathbb{R}^n}\frac{e^{-\frac{|x-z|^2}{4} \coth \tau} e^{-\frac{|x+z|^2}{4} \tanh \tau}}{(2 \pi \sinh 2\tau)^{n/2}}
g(z)\,dz,
\end{align*}
see \cite{Stempak},  \cite{Thangavelu}.
The operator $\partial_t$ in \eqref{oscilador} is taking care of the 
past, in other words its heat semigroup is given by $e^{-\tau\partial_t}\varphi(t) = \varphi(t-\tau).$
Hence for functions $f\in \mathcal{C}_{L^p(\mathbb{R}^n)}^1(\mathbb{R})$ we have $e^{-\tau \LL}f(t,x)  = e^{-\tau \mathcal{H}} \Big(e^{-\tau\partial_t} f(t,\cdot)\Big)(x)$, moreover 
\begin{align}\label{MehlerK2}
e^{-\tau \LL}f(t,x)  &=
e^{-\tau \mathcal{H}}(f(t-\tau,\cdot))(x) =  \int_{\mathbb{R}^n}\frac{e^{-\frac{|x-z|^2}{4} \coth \tau} e^{-\frac{|x+z|^2}{4} \tanh \tau}}{(2 \pi \sinh 2\tau)^{n/2}}
f(t-\tau,z)\,dz. 
\end{align}
The Fourier-Hermite transform of a function $f\in L^1(\R^{n+1})$ can be defined as
\begin{equation}\label{Fourier}
\F(f) (\rho,\mu) = \int_{\R^{n+1}} f(t,x) e^{-i\rho t} h_\mu(x) dt dx, \, \,\rho \in \R, \, \mu\in \N_0^n.
\end{equation}
Where    $h_{\mu}(x) = \prod_{j=1}^n h_{\mu_j}(x_j) \text{,}\ \ \ \text{$x=(x_1,\dots,x_n) \in\R^n$.} $ For  $k\in \N$, $h_k$ is the Hermite function defined by 
\begin{equation*}
h_k(t)= \frac{(-1)^k}{(2^k k! \pi^{1/2})^{1/2}}\, H_k(t)\, e^{-t^2/2}\text{,}\ \ \ \text{$t\in\R$.}
\end{equation*}
 Here
 $H_k$ denotes the Hermite polynomial of degree $k$ (see \cite{Thangavelu}). These functions are eigenvectors of the Hermite operator $\HH$. In fact
$
\HH h_\mu = (2|\mu|+n)\,h_\mu.
$
Consequently  for functions $f\in L^1(\mathbb{R}^{n+1})$ we have 
\begin{equation}\label{transf-heat}
\F(e^{-\tau \LL}f )(\rho,\mu) = e^{-\tau(i\rho+2|\mu |+n)}\F(f)(\rho,\mu),\,  \rho \in \R,\, \mu \in \N^n.
\end{equation}
Given $z \in \C$ with $\Re z \ge 0$, by analytic continuation  it can be seen that 
$$e^{-t \sqrt{z}} = \frac{y}{2\sqrt{\pi}}\int_0^\infty e^{-y^2/4\tau} e^{-\tau z}  \frac{d\tau}{\tau^{3/2}}.$$
Hence for $f\in L^1(\R^{n+1})$ we have 
$$ e^{-y\sqrt{i\rho+2|\mu |+n}}\mathcal{F}(f)(\rho,\mu) = \frac{y}{2\sqrt{\pi}}\int_0^\infty e^{-y^2/4\tau} e^{-\tau i\rho+2|\mu |+n}\mathcal{F}(f)(\rho,\mu)  \frac{d\tau}{\tau^{3/2}}.$$This last expression can  be written as 
$$\F(e^{-y\sqrt{\LL}}f)(\rho, \mu)= \frac{y}{2\sqrt{\pi}}\int_0^\infty e^{-y^2/4\tau} \F(e^{-y \LL} f)(\rho,\mu) \frac{d\tau}{\tau^{3/2}}. $$
The Fourier transform defined in \eqref{Fourier} is an isometry in $L^2(\R^{n+1})$ and in particular 
we have , in the $L^2(\mathbb{R}^{n+1})$ sense
\begin{align}\label{Poisson(n+1)}
 \mathcal{P}_yf(t,x) = e^{-y \sqrt{\LL}}f(t,x) = \frac{y}{2\sqrt{\pi}}\int_0^\infty e^{-y^2/4\tau} e^{-\tau \mathcal{\LL}}f(t,x)  \frac{d\tau}{\tau^{3/2}}. 
\end{align}
For functions $f$ good enough, formulas \eqref{MehlerK2} and \ref{Poisson(n+1)} give  the following pointwise expression  
\begin{align}\label{Poisson(n+1)2}
\mathcal{P}_yf(t,x)= \frac{y}{2\sqrt{\pi}}\int_0^\infty \int_{\mathbb{R}^n} e^{-y^2/4\tau} \frac{e^{-\frac{|x-z|^2}{4} \coth \tau} e^{-\frac{|x+z|^2}{4} \tanh \tau}}{(2 \pi \sinh 2\tau)^{n/2}}
f(t-\tau,z)  \,dz\frac{d\tau}{\tau^{3/2}}, \, x\in \mathbb{R}^n,  \, t \in \R. 
\end{align}
On the other hand  
\begin{equation*} ye^{-y^2/4\tau} \frac{e^{-\frac{|x-z|^2}{4} \coth \tau} e^{-\frac{|x+z|^2}{4} \tanh \tau}}{(2 \pi \sinh 2\tau)^{n/2}\tau^{3/2}}\chi_{\{\tau>0\}}\le C
 \frac{y}{\tau^{1/2}} \frac{e^{-\frac{y^2}{4\tau}}}{\tau} \frac{e^{-\frac{|x-z|^2}{4\tau}}}{\tau^{n/2}}\chi_{\{\tau>0\}} = \Phi_y(\tau,x-z).
 \end{equation*}
 As $\Phi_y$ belongs to $L^1(\R^{n+1})$,   the formula \eqref{Poisson(n+1)2}, defining the {\it Parabolic Poisson Hermite integral} , remains  valid for any  $f\in L^p(\R^{n+1}), \, 1\le p\le \infty, \, \, (x,t) \in \R^{n+1}.$
 Moreover this integral satisfies a {\it Parabolic Hermite Laplace equation} as the following Proposition shows. 
 \begin{proposition}\label{PropEqPoisson}
Assume $f\in L^\infty(\R^{n+1}).$ Then $\mathcal{P}_yf(t,x)$ satisfies the equation 
\begin{equation}\label{EqPoisson}
\partial_y^2 \mathcal{P}_yf(t,x) -\LL\mathcal{P}_yf(t,x) =0, \:\; (x,t) \in \R^{n+1}.
\end{equation}
\end{proposition}
\begin{proof}
We observe that 
\begin{align*}
\Big|\partial_{y^2} \Big(ye^{-y^2/4\tau}& \frac{e^{-\frac{|x-z|^2}{4} \coth \tau} e^{-\frac{|x+z|^2}{4} \tanh \tau}}{(2 \pi \sinh 2\tau)^{n/2}\tau^{3/2}}\chi_{\{\tau>0\}} \Big)\Big| \\ &+  \Big|\Delta_x \Big(ye^{-y^2/4\tau} \frac{e^{-\frac{|x-z|^2}{4} \coth \tau} e^{-\frac{|x+z|^2}{4} \tanh \tau}}{(2 \pi \sinh 2\tau)^{n/2}\tau^{3/2}}\chi_{\{\tau>0\}} \Big)\Big|\\ & + \Big|\partial_\tau \Big(ye^{-y^2/4\tau} \frac{e^{-\frac{|x-z|^2}{4} \coth \tau} e^{-\frac{|x+z|^2}{4} \tanh \tau}}{(2 \pi \sinh 2\tau)^{n/2}\tau^{3/2}}\chi_{\{\tau>0\}} \Big)\Big|\\& \le
 \frac{C}{\tau} \frac{e^{-\frac{y^2}{4\tau}}}{\tau} \frac{e^{-\frac{|x-z|^2}{4\tau}}}{\tau^{n/2}}\chi_{\{\tau>0\}}.
\end{align*}
Hence, for $y>0$ and $|x-z|>0$, the function $ye^{-y^2/4\tau} \frac{e^{-\frac{|x-z|^2}{4} \coth \tau} e^{-\frac{|x+z|^2}{4} \tanh \tau}}{(2 \pi \sinh 2\tau)^{n/2}\tau^{3/2}}\chi_{\{\tau>0\}}$ is smooth in all its variables. In particular  we can write
 \begin{align*}
\mathcal{P}_yf(t,x)= \frac{y}{2\sqrt{\pi}}\int_{\mathbb{R}}\int_{\mathbb{R}^n} e^{-y^2/4(t-\tau)} \frac{e^{-\frac{|x-z|^2}{4} \coth (t- \tau)} e^{-\frac{|x+z|^2}{4} \tanh(t- \tau)}}{(2 \pi \sinh 2(t-\tau))^{n/2}}
f(\tau,z)  \,dz\,  \chi_{\{t-\tau >0\}}\frac{d\tau}{(t-\tau)^{3/2}}.
\end{align*}
The above estimates also  show that we can interchange the derivatives with the integral for for $y>0$ and $|x-z|>0.$ Hence the Proposition follows since the kernel of this last integral satisfies the equation \eqref{EqPoisson}.
\end{proof}

\begin{remark} The proof of the previous Lemma also shows that for functions $f\in L^\infty(\mathbb{R}^{n+1}) $ we can write

	\begin{align}\label{solution}
	\mathcal{P}_yf(t,x) &= \int_{\mathbb{R}^{n+1}} \mathcal{P}_y(\tau,x,z) f(t-\tau,x-z)  dz d\tau 
\\  \nonumber & = \frac{y}{2\sqrt{\pi}}\int_{\mathbb{R}}\int_{\mathbb{R}^n}\frac{e^{-\frac{|z|^2}{4}\coth\tau}e^{-\frac{|2x-z|^2}{4}\tanh\tau}e^{-\frac{y^2}{4\tau}}}{(2\pi\sinh(2\tau))^{n/2}\tau^{3/2}}f(t-\tau,x-z)dz\, \chi_{\{\tau >0\}}\, d\tau.
	\end{align}
\end{remark}

\

As we have noticed in \eqref{transf-heat}, the infinitesimal generator, $\LL$, of the
semigroup $e^{-\tau\LL}$ is not positive. This forced us to use some complex variable technique in order to give a sense to the powers of the operator $\LL$.  Given a non necessarily  positive operator $\mathbb{L}$, formulas to define $\mathbb{L}^{\pm \alpha}$, where $0<\alpha <1$, were considered in   \cite{Bernardis},  \cite{ST}  and \cite{ST3}.

Given $0< \beta$ , we recall the following two integrals related with the Gamma function:
\begin{equation}\label{defi1}
C_\beta = \int_0^\infty e^{-t} t^\beta \,\frac{dt}{t}, \quad  
c_\beta = \int_0^\infty \big(e^{-t}-1\big)^{[\beta]+1}\,\frac{dt}{ t^{1+\beta}}.
\end{equation}
It is well known that $C_\beta = \Gamma(\beta) $ for all $0<\beta$ and $c_\beta = \Gamma(-\beta)$ for   $0< \beta <1$.  The following Lemma was proved in \cite{Bernardis}.
 
\begin{lemma}\label{lem:gamma function}
Let $0<\beta<1$ and $-\pi/2\leq \varphi_0\leq\pi/2$. Consider the ray
in the complex plane $\mathrm{ray}_{\varphi_0}:=\{z=re^{i\varphi_0}:0<r<\infty\}$. Then
\begin{equation*}
\Gamma (\beta) = \int_{ {\rm ray}_{\varphi_0} } e^{-z} z^\beta \,\frac{dz}{z},\quad  \hbox{and} \quad 
\Gamma(-\beta) = \int_{{\rm ray}_{\varphi_0}} ( e^{-z}-1)\,\frac{dz}{z^{1+\beta}}.
\end{equation*}
\end{lemma}
For $0<\beta <1$, the absolutely convergent integrals in \eqref{defi1} can be interpreted  as integrals of the functions $F(t) = e^{-t} t^{\beta-1}$ and $G(t) = (e^{-t}-1)/ t^{1+\beta}$
along the ``complex''  path $\{z=t:  0<t< \infty\}$.
The proof of the Lemma is based in the Cauchy Integral Theorem applied to the functions $F(z) = e^{-z} z^{\beta-1}$ and $G(z) = (e^{-z}-1)/ z^{1+\beta}$. Both functions are analytic for $z \neq 0$.   For the integrals defined in  \eqref{defi1} we could state a parallel Lemma to \ref{lem:gamma function}, by choosing  $H(z) = (e^{-z}-1)^{[\beta]+}/ z^{1+\beta}$. The proof follows the same steps.  We leave the details to the reader. We have the  following Corollary.

\begin{corollary}\label{formulas} Let $ \beta >0 $ and $\lambda$ a complex number with $\Re \lambda \ge 0$.  Then
\begin{equation*}
\lambda^{-\beta} = \frac1{\Gamma(\beta)} \int_0^\infty e^{\lambda t }t^{\beta}\,\frac{dt}{t},\quad
\hbox{and}\quad \lambda^{\beta} = \frac1{c_\beta} \int_0^\infty
(e^{\lambda t}-1)^{[\beta]+1}\,\frac{dt}{t^{1+\beta}}.\end{equation*}
\end{corollary}

We use the last  Corollary to define 
define the negative and positive fractional powers of the operator $\LL$ as 
\begin{equation*}
\mathcal{L}^\beta f(t,x)=\frac{1}{c_{2\beta}}\INT\left( e^{-\tau \mathcal{L}^{1/2}}-I\right)^{[2\beta]+1}f(t,x) \frac{d\tau}{\tau^{1+2\beta}},
\end{equation*}
where $c_{2\beta}=\INT\left( e^{-\tau}-1\right)^{[2\beta]+1} \frac{d\tau}{\tau^{1+2\beta}}$.
Also, for $\beta>0$, 
$$
\mathcal{L}^{-\beta}f(t,x)=\frac{1}{\Gamma(2\beta)}\INT e^{-\tau \mathcal{L}^{1/2}}f(t,x)\frac{d\tau}{\tau^{1-2\beta}}.
$$

Observe that for good enough functions
\begin{eqnarray*}
\mathcal{F}(\LL^{\pm \beta} f)(\rho,\mu) = (i\rho+2\mu+n)^{\pm \beta} \mathcal{F}( f)(\rho,\mu), \, \rho \in \mathbb{R}, \hbox{  and  } \mu \in \mathbb{N}^n.
\end{eqnarray*}

\subsection{Elliptic Hermite setting} \label{EllipticSeting}

\

Given  $g\in L^\infty(\R^n)$, consider the function 
 $f(t,x)= g(x)$, then formula \eqref{Poisson(n+1)2}
  becomes 
  \begin{align}\label{Poissonn}
\mathcal{P}_yf(t,x)= \frac{y}{2\sqrt{\pi}}\int_0^\infty \int_{\mathbb{R}^n} e^{-y^2/4\tau} \frac{e^{-\frac{|x-z|^2}{4} \coth \tau} e^{-\frac{|x+z|^2}{4} \tanh \tau}}{(2 \pi \sinh 2\tau)^{n/2}}
g(z)  \,dz\frac{d\tau}{\tau^{3/2}} = P_yg(x).
\end{align}
Where  $P_yg(x)$ is  the Poisson semigroup associated to the operator $\HH= -\Delta_x+|x|^2.$ The thoughts developed along this section show that: 
\begin{itemize}
\item For functions $g \in L^\infty(\mathbb{R}^n)$, $P_yg(x)$ satisfies the equation  $\displaystyle \partial_y^2 P_yg(x) - \HH P_y g(x) =0, \, x \in \R^{n+1}.$
\item Identities \eqref{solution} and \eqref{Poissonn} give
 that $\displaystyle\int_{\R}\mathcal{P}_y(\tau,x,z)d\tau =P_y(x,z)$, for all $x,z\in\R^n$, where $\mathcal{P}_y$ is the Poisson kernel associated to $\mathcal{L}$ and $P_y$ is the Poisson kernel associated to the harmonic oscillator, $\mathcal{H}$.
\item Let $\beta >0$, for $g$ good enough, \begin{equation*}
\mathcal{H}^\beta g(x)=\frac{1}{c_{2\beta}}\int\left( e^{-\tau \mathcal{H}^{1/2}}-I_d\right)^{[2\beta]+1}g(x) \frac{d\tau}{\tau^{1+2\beta}},
\end{equation*}
is well defined and $\widehat{\mathcal{H}^\beta g}(\mu) = (2|\mu|+n)^{\beta} \hat{g}(\mu), \, \mu \in \mathbb{N}^n$, with $\hat{ g}(\mu) = \int_{\R^n} g(x) h_\mu(x) dx.$
\item Let $\beta >0$, for  good enough functions $g$,   $$
\mathcal{H}^{-\beta}g(x)=\frac{1}{\Gamma(2\beta)}\INT e^{-\tau \mathcal{L}^{1/2}}g(x)\frac{d\tau}{\tau^{1-2\beta}}
$$
is well defined and $\widehat{\mathcal{H}^{-\beta} g}(\mu) = (2|\mu|+n)^{-\beta} \hat{g}(\mu), \, \mu \in \mathbb{N}^n.$

\end{itemize}

\section{Coincidence of Parabolic Hermite-Zygmund   with Parabolic Hermite-H\"older spaces.}\label{proofs}

 We shall begin by recalling the following definition, it can be found in \cite{ST2}.

\begin{definition}\label{DefPablo}[Hermite H\"older spaces] Let $0<\alpha < 1$. We consider the space of functions  
	$$C_{\mathcal{H}}^{\alpha}(\R^n)= \{ f: (1+|\cdot|)^\alpha f(\cdot) \in L^\infty(\R^{n}),
	\hbox{ and }  \|f(\cdot+z) - f(\cdot)\|_{L^\infty(\R^n) } \le A |z|^\alpha\}$$
	with   associated norm
	$$
	\|f\|_{C_{\mathcal{H}}^{\alpha}}=[f]_{M^\alpha}+[f]_{C_{\mathcal{H}}^{\alpha}}.
	$$
	Where $[f]_{M^\alpha}=\|(1+|\cdot|)^\alpha f(\cdot)\|_\infty$ and $\displaystyle [f]_{C_{\mathcal{H}}^{\alpha}}=\sup_{|z|>0}\frac{\|f(\cdot+z)-f(\cdot)\|_\infty}{|z|^\alpha}$.
	
	\noindent For $\alpha >1$ and not integer, we say that  $f\in C_{\mathcal{H}}^{ \alpha}(\R^n)$, if  there exist the derivatives of order $[\alpha]$   and  the norm
	$$
	\|f\|_{ C_{\mathcal{H}}^{\alpha}}:=[f]_{M^{\alpha-[\alpha]}}+\displaystyle\sum_{\substack{1\le |i_1|,\dots,|i_m|\le n\\1\le m\le [\alpha]}}[A_{i_1}\dots A_{i_m}f]_{M^{\alpha- [\alpha]}}+\sum_{1\le|i_1|,\dots,|i_[\alpha]|\le n}[A_{i_1}\dots A_{i_{[\alpha]}}f]_{C^{\alpha-[\alpha]}_{\mathcal{H}}},
	$$
	is finite.  
\end{definition}
Some parabolic H\"older spaces were considered by N. Krylov, see \cite{Krylovbook}. Namely
\begin{itemize}
\item[(i)] Let $0< \alpha <1$, $C^{\alpha/2,\alpha}$ was   defined as   the set of bounded functions such that
 $$	[f]_{C^{\alpha/2,\alpha}} = \sup_{(\tau,z)\neq (0,0)} 
	 \frac{\|f(\cdot-\tau,\cdot-z)-f(\cdot,\cdot)\|_{L^\infty(\R^{n+1})}}{(|\tau|^{1/2}+|z|)^{\alpha}} < \infty.$$ 
\item[(ii)] For $1<\alpha<2$,  $f\in C^{\alpha/2,\alpha}$ if $\partial_{x_i} f\in C^{\alpha/2-1/2,\alpha-1}$  and $f(\cdot,x)\in C^{\alpha/2}(\R)$ uniformly on $x.$
\item[(iii)] Let  $0< \alpha <1$, $C^{1+ \alpha/2,2+\alpha}$ if
$\partial^2_{x_i} f$ and $\partial_t f$ belong to $ C^{\alpha/2,\alpha}.$
\end{itemize}
These Krylov's definitions together with Definition \ref{DefPablo}  drive us to consider the following definition.\begin{definition}\label{ParabolicHermite} [Parabolic Hermite H\"older spaces]
	
	\begin{itemize}
		\item 	Let $0<\alpha<1$. We say that $f\in C_{t,\mathcal{H}}^{\alpha/2,\alpha}$ if $f\in C^{\alpha/2,\alpha}$ and 
	$$
	[f]_{M^\alpha}=\displaystyle\sup_{(t,x)\in\mathbb{R}^{n+1}}(1+|x|)^{\alpha}|f(t,x)|<\infty,
	$$
 In this case, $\|f\|_{C_{t,\mathcal{H}}^{\alpha/2,\alpha}}=[f]_{\mathcal{M}^\alpha}+[f]_{C_{t,\mathcal{H}}^{\alpha/2,\alpha}}$.
 
 \item For $1<\alpha<2$,  $f\in C_{t,\mathcal{H}}^{\alpha/2,\alpha}$ if $A_{\pm i}f\in C_{t,\mathcal{H}}^{\alpha/2-1/2,\alpha-1}$  and $f(\cdot,x)\in C^{\alpha/2}(\R)$ uniformly on $x$.

\item For $2< \alpha<3$  we say that a function $f\in  C_{t,\mathcal{H}}^{\alpha/2,\alpha}$,  if the functions $A_{\pm i}A_{\pm j} f$  and the function $\partial_t f$ belong to $ C^{\alpha/2-1,\alpha-2}_{t,\mathcal{H}}.$   \end{itemize}

	\end{definition}

In the next result we will show that the functions in $C_{t,\mathcal{H}}^{\alpha/2,\alpha}$, {$0<\alpha<1$}, can be taken to be continuous, so the inequality  $|f(t-\tau,x+z)-f(t,x)|\leq C(\tau^{1/2}+|z|)^\alpha$ holds for every $x\in\R^n$, $t\in\R.$

\begin{proposition}\label{convergencia2}
	{For $0<\alpha<1$,} every $f\in C_{t,\mathcal{H}}^{\alpha/2,\alpha}(\R^{n+1})$ can be modified on a set of measure zero so that it becomes continuous.
\end{proposition}

\begin{proof}
	Let $f\in C_{t,\mathcal{H}}^{\alpha/2,\alpha}(\R^{n+1})$. We will follow the ideas in  Stein \cite[page 142]{Stein}. By the hypothesis on $f$,  Lemma \ref{lemma5B} (i) and Lemma \ref {'lemma4'} $(3)$   we have
	\begin{multline*}
	|\mathcal{P}_yf(t,x)-f(t,x)| \\ \le  \left|\int_{\R^{n+1}} \mathcal{P}_y(\tau,x,z)(f(t-\tau,x-z)-f(t,x))d\tau dz\right| + \Bigg|f(t,x)\left(\int_{\R^{n+1}}\mathcal{P}_y(\tau,x,z)d\tau dz-1\right) \Bigg|\\ \le
[f]_{C_{t,\mathcal{H}}^{\alpha/2,\alpha}} \Big(\int_{\R^n}\int_0^\infty\frac{ye^{-\frac{y^2+|z|^2}{c\tau}}(\tau^{1/2}+|z|)^\alpha}{\tau^{\frac{n+3}{2}}}d\tau dz \Big)	  + \|f\|_\infty  \bigg|e^{-y \sqrt{ \LL}}1(t,x)-1\bigg|  \le
	C\| f\|_{C_{t,\mathcal{H}}^{\alpha/2,\alpha}} y^\alpha.
	\end{multline*}
	In particular, 
	we conclude that $\mathcal{P}_yf$ converges uniformly to $f$ as $y$ goes to zero.  As $\mathcal{P}_y f$ is continuous,   $f$ can be taken to be continuous.
\end{proof}
Now we shall show that, for $0<\alpha <1$, the pointwise Definition \ref{ParabolicHermite} is equivalent to the 
Definition \ref{ZygmundPar} given by using  of Poisson semigroup.

\begin{theorem}\label{caracHolder2} [$0<\alpha <1$, Parabolic Hermite-H\"older $=$ Parabolic Hermite-Zygmund]
	Let $0<\alpha<1$.  Then $$ C_{t,\mathcal{H}}^{\alpha/2,\alpha} =  {	\Lambda^\alpha_{\mathcal{L}}} ,$$ with equivalence of norms.
\end{theorem}

\begin{proof}
	For  $f\in C_{t,\mathcal{H}}^{\alpha/2,\alpha}(\R^{n+1})$, we write
\begin{multline*}
	y\partial_y \mathcal{P}_y f(t,x)=\int_{\mathbb{R}^{n+1}}y\partial_y \mathcal{P}_y(\tau,x,z)(f(t-\tau,x-z)-f(t,x))d\tau dz\\ +f(t,x)\int_{\mathbb{R}^{n+1}}y\partial_y \mathcal{P}_y(\tau,x,z)d\tau dz=I_1+I_2.
\end{multline*}
	By   Lemma \ref{lemma5B} (i)  we have 
	\begin{align*}
	|I_1|&\leq \int_{\mathbb{R}^{n+1}} |y\partial_y \mathcal{P}_y(\tau,x,z)|
	|f(t-\tau,x-z)-f(\tau,x)|dz 
\\&	\le C\|f\|_{C^{0,\alpha/2,\alpha}_{\mathcal{H}}}\int_{\mathbb{R}^n}\int_0^\infty
	\frac{ye^{-\frac{y^2+|z|^2}{c\tau}}(\tau^{1/2}+|z|)^\alpha }{\tau^{\frac{n+3}{2}}}d\tau dz
	\le C \|f\|_{C^{\alpha/2,\alpha}_{t,\mathcal{H}}} y^{\alpha}.
	\end{align*}
	Regarding $I_2$,   as $\displaystyle \int_0^\infty  y\partial_y(y e^{-y^2/4\tau}) \frac{d\tau}{\tau^{3/2}} = 0$ we can write 
	\begin{align*}
	|I_2 |&= \Big|f(t,x) \frac1{2\sqrt{\pi}}\int_0^\infty  y\partial_y(y e^{-y^2/4\tau}) \Big( e^{-\tau{\LL}}1(t,x) -1\Big)  \frac{d\tau}{\tau^{3/2}}\Big|
	\le C \|f\|_{C^{\alpha/2,\alpha}_{t,\mathcal{H}}}y^\alpha.
	\end{align*}
	Where in the last inequality we have used  Lemma \ref{'lemma4'} (3).
	
	Conversely, suppose that $f\in	\Lambda^\alpha_{\mathcal{L}}$. We can write
\begin{multline*}
	f(t+\tau,x+z)-f(t,x)\\=(\mathcal{P}_yf(t+\tau,x+z)-\mathcal{P}_y f(t,x))+(f(t+\tau,x+z)-\mathcal{P}_yf(t+\tau, x+z))+ (\mathcal{P}_yf(t,x)-f(t,x)).
\end{multline*}
	Let  $y=\tau^{1/2}+|z|$. For the second summand  we have  
	\begin{multline*}
	\Big \|f(t+\tau,x+z)-\mathcal{P}_yf( t+\tau,x+z)\Big\|_\infty=\left\| -\int_0^y\frac{\partial \mathcal{P}_{y'}f(t+\tau,x+z)}{\partial y'}dy'\right\|_\infty
	\\ \leq C\|f\|_{C^{0,\alpha/2,\alpha}_{\mathcal{H}}}\int_0^y y'^{-1+\alpha}dy' =C\|f\|_{	\Lambda^\alpha_{\mathcal{L}}}y^\alpha=C\|f\|_{	\Lambda^\alpha_{\mathcal{L}}}(\tau^{1/2}+|z|)^\alpha.
	\end{multline*}
	 A similar estimate can be performed for  the third summand.
	On the other hand by the Mean Value Theorem and Lemma \ref{est2}, we have  
	\begin{align}\label{reparto}
	\nonumber |\mathcal{P}_yf(t+\tau,x+z)-\mathcal{P}_yf(t,x)|&\leq|\mathcal{P}_yf(t+\tau,x+z)-\mathcal{P}_yf(t+\tau,x)|+|\mathcal{P}_yf(t+\tau,x)-\mathcal{P}_yf(t,x)|
	\nonumber  \\&\le|\nabla_x \mathcal{P}_yf(t+\tau,x+\theta z)||z|+|\partial_t\mathcal{P}_yf(t+\lambda\tau,x)||\tau|.	\end{align}

	We observe that by the semigroup property, integration by parts  and Lemma \ref{lemma5B} (ii), we have 
	\begin{eqnarray*} \Big| \partial_{x_i} \partial_y\mathcal{P}_yf(t,x) \Big|&=& \Big|\int_{\R^{n+1}} \partial_{x_i} \mathcal{P}_{y/2}(\tau,x,z)\partial_y\mathcal{P}_y f(t-\tau,x-z)\big|_{y/2}d\tau dz \\
	&&  + \int_{\R^{n+1}} \mathcal{P}_{y/2}(\tau,x,z)\partial_{x_i}\partial_y\mathcal{P}_y f(t-\tau,x-z)\big|_{y/2}d\tau dz \Big | \\ &\le & \int_{\R^{n+1}} \Big| (\partial_{x_i}+\partial_{z_i})\mathcal{P}_{y/2}(\tau,x,z)\partial_y \mathcal{P}_y f(t-\tau,x-z)\big|_{y/2}\Big |d\tau dz  \le Cy^{-2+\alpha}.
	\end{eqnarray*}
	
	Hence as by Lemma \ref{lemma5B} we have  
	$| \partial_{x_i} \mathcal{P}_y f(t,x)| \le C/y$, then $$ \Big| \partial_{x_i} \mathcal{P}_yf(t,x)\Big|=\Big|\int_y^\infty \partial_{y'} \partial_{x_i} \mathcal{P}_{y'} f(t,x) dy' \Big| \le Cy^{-1+\alpha}.$$
	The derivative  $\Big| \partial_{t} \mathcal{P}_yf(t,x) \Big|$ can be handled in a parallel way, this time  using point (iv) of Lemma  \ref{lemma5B}, we get  $\Big| \partial_{t} \mathcal{P}_yf(t,x) \Big|\le Cy^{-2+\alpha}$. Then going back to (\ref{reparto})   we have 
	$$	|\mathcal{P}_yf(t+\tau,x+z)-\mathcal{P}_yf(t,x)|	\leq C\|f\|_{	\Lambda^\alpha_{\mathcal{L}}}(\tau^{1/2}+|z|)^{\alpha}.$$
	
	Finally we shall see  that $(1+|x|)^\alpha f\in L^{\infty}(\R^{n+1}).$ Given $k\in \mathbb{N}$,  a direct application of Lemma \ref{lemma5B} (ii) gives $\Big| \partial^k_y x^\gamma_i \mathcal{P}_yf(t,x)\Big| \le C \|f\|_\infty y^{-(k+\gamma+s)}, s >0.$ Moreover by the semigroup property we have $\partial^k_y x^k_i \mathcal{P}_yf(t,x) = \int_{\R^{n+1}} x^k_i \mathcal{P}_{y/2}(\tau,x,z)\partial^k_y\mathcal{P}_y f(t-\tau,x-z)\big|_{y/2}d\tau dz.$ For  $k=[\alpha]+1$, the hypothesis and Lemma \ref{lemma5B} (ii) give    	
	$\Big| \partial^k_y x^k_i \mathcal{P}_yf(t,x)\Big|  \le C \|f\|_{	\Lambda^\alpha_{\mathcal{L}}} y^{-(k+k-\alpha)}.$ Then an iterated integration gives 
	$\Big|  x^k_i \mathcal{P}_yf(t,x)\Big|  \le C \|f\|_{	\Lambda^\alpha_{\mathcal{L}}} y^{-(k-\alpha)}.$ 	Now for $|x|  >1$  and $0<\alpha <1$ we have
	\begin{align*}
	\nonumber	|x|^\alpha|f(t,x)|&\le |x|^\alpha\sup_{0<y<\frac{1}{|x|}}|\mathcal{P}_yf(t,x)|\le |x|^\alpha\sup_{0<y<\frac{1}{|x|}}\left(|\mathcal{P}_yf(t,x)-\mathcal{P}_{\frac{1}{|x|}}f(t,x)|+|\mathcal{P}_{\frac{1}{|x|}}f(t,x)|\right)\\
	&\le |x|^\alpha\sup_{0<y<\frac{1}{|x|}}\left| \int_y^{\frac{1}{|x|}}\partial_{z_1} \mathcal{P}_{z_1}f(t,x) d{z_1}  \right|+ C\|f\|_{	\Lambda^\alpha_{\mathcal{L}}} \\ &\le
	 |x|^\alpha\sup_{0<y<\frac{1}{|x|}}\left| \int_y^{\frac{1}{|x|}}z_1^{-(1-\alpha)}d{z_1}  \right|+ C\|f\|_{	\Lambda^\alpha_{\mathcal{L}}} \le C.  	\end{align*}

\end{proof}

\subsection{Elliptic Hermite setting} \label{EllipticSeting2}

\

   Let $g$ an $L^\infty(\mathbb{R}^n)$ function. Consider, as in Remark \ref{EllipticSeting}  the function $f(t,x) = g(x)$.  It is clear that if $g\in C^\alpha_{\HH}$ if and only if $f\in C^{\alpha/2,\alpha}_{t,\HH}$.   Moreover, as $\mathcal{P}_y f(t,x) = P_yg(x)$,  $g\in \Lambda^\alpha_{\HH}$ if and only if $f\in \Lambda^\alpha_{\LL}$.  Hence, for $0<\alpha <1$,  Proposition \ref{convergencia2} and Theorem \ref{caracHolder2}  have as consequences the continuity of the functions $g\in \Lambda^\alpha_{\HH}$ and the identity $ \Lambda^\alpha_{\HH}=  C_{\HH}^\alpha(\mathbb{R}^n).$

\section{Proofs of Theorems \ref{teo151} and \ref{teo16}. Coincidence with H\"older spaces for $ \alpha >1 .$}
\label{Main}

\begin{remark}\label{varioslambdas2}
	Observe that for bounded functions $f$, Lemma \ref{lemma5B} assures that \newline
	$\displaystyle \left\|\partial_y^k \mathcal{P}_y f\right\|_{L^\infty(\mathbb{R}^{n+1})} \le C\|f\|_\infty y^{-k}.$ Therefore we can assume in Definition \ref{ZygmundPar} that $y<1$.
\end{remark}

\begin{lemma}\label{'lemma16B'}  
	Let $f\in L^\infty(\R^{n+1})$, $\alpha>0,$  and $k, l$ integers bigger than $\alpha$. Then, for $y>0$, the following conditions are equivalent:
	\begin{itemize}
		\item[(a)]
		$\left\|\partial_y^k \mathcal{P}_y f \right\|_{L^\infty(\mathbb{R}^{n+1})}\leq A_k y^{-k+\alpha}
		$
		\item[(b)]
		$
		\left\|\partial_y^l \mathcal{P}_y f \right\|_{L^\infty(\mathbb{R}^{n+1})}\leq A_l  y^{-l+\alpha},
		$
	\end{itemize}
	where $A_k$ and $A_l$ are  positive constants with $A_k \sim A_l$.
\end{lemma}

\begin{proof}

Let $l=k+1$. By using the semigroup property we have
\begin{eqnarray*}
\partial_y^l \mathcal{P}_yf(t,x) = \int_{\R^{n+1}}\partial_y \mathcal{P}_{y}(\tau,x,z)\Big|_{y/2} \partial_y^k \mathcal{P}_{y}f(t-\tau,x-z)\Big|_{y/2} d\tau dz.
\end{eqnarray*}
By Lemma \ref{lemma5B} $(ii)$ we get $(a) \implies (b)$.  
For the converse,  Remark \ref{varioslambdas2} allows the integration
$\displaystyle  \partial_y^k \mathcal{P}_y f(t,x) = \int_y^\infty \partial_z^{k+1} \mathcal{P}_z f(t,x)dz, $ that gives the result.
\end{proof}

\begin{corollary}\label{encajadas2}
		Let $\alpha>0$. If $f\in 	\Lambda^\alpha_{\mathcal{L}}$, then for every $0<\beta<\alpha,$ $f\in \Lambda^{\beta}_{\LL}$.
	\end{corollary}
For the proof of this Corollary observe that, given $ k_\alpha= [\alpha]+1$, we have (for $y<1$) 
$$\Big\|\partial^{k_\alpha}_yP_yf\Big\| \le A_{k_\alpha}  \|f\|_{	\Lambda^\alpha_{\mathcal{L}}}y^{-k_\alpha +\alpha} \le A_{k_\alpha} \|f\|_{	\Lambda^\alpha_{\mathcal{L}}} y^{-k_\alpha +\beta}. $$ Then, Lemma \ref{'lemma16B'} gives the result.
\begin{lemma}\label{est2} 
	Let $\alpha>0$, $f\in 	\Lambda^\alpha_{\mathcal{L}}$ and $k=[\alpha]+1$.  
	\begin{enumerate}
		\item For every $\gamma \ge 0$ and $m,j\in \N_0$ such that $\gamma+m+j\ge k$ there exists a constant $C_{\gamma, m,j } $ such that $\||\cdot|^\gamma \partial_y^m\partial_{x_i}^j\mathcal{P}_yf \|_\infty \le C_{\gamma,m,j}  \|f\|_{	\Lambda^\alpha_{\mathcal{L}}}  y^{-(\gamma+m+j)+\alpha}$.
		\item For every $m$ such that   $m+2 \ge k$, there exists a constant $C_{m}$ such that  $\|\partial_y^{m}\partial_t \mathcal{P}_yf\|_\infty\le Cy^{-(m+2)+\alpha}$.
	\end{enumerate}
\end{lemma}
\begin{proof}
	 Observe that the case $\gamma=j=0$ follows from the definition of the space $	\Lambda^\alpha_{\mathcal{L}}$, so we will exclude it in the following. Let us analyze the case when $m\ge k$. By the semigroup property and integration by parts  we have
	\begin{align*}
	\Big||x|^\gamma \partial_y^{m}\partial_{x_i}^j\mathcal{P}_yf(t,x)\Big|&=\Big||x|^\gamma\partial_{x_i}^j\int_{\R^{n+1}} \mathcal{P}_{y/2}(\tau,x,z)\partial_y^{m}\mathcal{P}_y f(t-\tau,x-z)\big|_{y/2}d\tau dz\Big|\\&=\Big||x|^\gamma \int_{\R^{n+1}}(\partial_{x_i}+\partial_{z_i})^j \mathcal{P}_{y/2}(\tau,x,z) \partial_y^{m}\mathcal{P}_y f(t-\tau,x-z)\big|_{y/2}d\tau dz\Big| \\
	&\le C \|\partial_y^{m}\mathcal{P}_y f\big|_{y/2}\|_\infty \sum_{p+q=j} 
	\int_{\R^{n+1}}|x|^\gamma|\partial_{z_i}^{p}\partial_{x_i}^q \mathcal{P}_{y/2}(\tau,x,z)|d\tau dz  \\&
	\le C_{\gamma,m,j}  \|f\|_{	\Lambda^\alpha_{\mathcal{L}}}y^{-(\gamma+m+j) +\alpha}.	\end{align*}
	In the last inequality we have use the hypothesis on $f$ and Lemma \ref{lemma5B} (ii) in each summand.  We have chosen $s=j+\gamma$ in the case $p-q+\gamma\le 0$. While in the case  $p-q+\gamma >0$, we choose  $s=  2q$.

Now we prove $(2)$  for $m\ge k$.  By the semigroup property, the hypothesis on $f$ and Lemma \ref{lemma5B} (iv)  we have 
\begin{align*}
\partial_y^{m}\partial_{t}\mathcal{P}_yf(t,x)&=\int_{\R^{n+1}}\partial_{\tau} \mathcal{P}_{y/2}(\tau,x,z)\partial_y^{m}\mathcal{P}_y f(t-\tau,x-z)\big|_{y/2}d\tau dz\\
&\le C \|\partial_y^{m}\mathcal{P}_y f\big|_{y/2}\|_\infty\int_{\R^{n+1}}\partial_{\tau}\mathcal{P}_{y/2}(\tau,x,z)d\tau dz\le C\|f\|_{	\Lambda^\alpha_{\mathcal{L}}}y^{-(m+2)+\alpha}.
\end{align*}

In both cases,  for $m<k$ we start from the above estimates for  the case $m=k$ and then we perform an $k-m$ iterated integration. 
\end{proof}

\begin{proposition}\label{taman}
	Let $\alpha>0$. If $f\in 	\Lambda^\alpha_{\mathcal{L}}$, then $|x|^\alpha f\in L^\infty(\R^{n+1})$. 
\end{proposition}
\begin{proof} If $\alpha$ is not an integer we can use the same argument as in the proof of Theorem \ref{caracHolder2}.

Let  $\alpha=1$,  by using the arguments in the proof of Theorem \ref{caracHolder2}, we can obtain 
$$\Big||x|^2 \mathcal{P}_{y}f(t,x)\Big|+ \Big||x| \partial_{v} \mathcal{P}_{v}f(t,x)\Big|_{v=\frac1{|x|}} \le 
 \|f\|_{\Lambda^{1}_{\LL}}y^{-1}.$$ 
	Then, by using that
	$\partial_{z_1} \mathcal{P}_{z_1}f(t,x) = -\int_{z_1}^\frac1{x} \partial^2_{z_2} \mathcal{P}_{z_2}f(t,x)dz_2 + \partial_{v} \mathcal{P}_{v}f(t,x)\Big|_{v=\frac1{|x|}}$, we have
	\begin{multline*}
||x| f(t,x) |	\\ \le |x|\sup_{0<y<\frac{1}{|x|}}\left| -\int_y^{\frac{1}{|x|}} \Big( \int_{z_1}^\frac1{|x|} \partial^2_{z_2} \mathcal{P}_{z_2}f(t,x)dz_2 + \partial_{v} \mathcal{P}_{v}f(t,x)\Big|_{v=\frac1{|x|}}  \Big)d{z_1}  \right| +\Big| |x| \mathcal{P}_{1/|x|}f(t,x)\Big| \\
	\le |x|\sup_{0<y<\frac{1}{|x|}}\int_y^{\frac{1}{|x|}}  \int_{z_1}^\frac1{|x|} \left| \partial^2_{z_2} \mathcal{P}_{z_2}f(t,x)\right|dz_2d{z_1} +\|f\|_{\Lambda^{1/2,1}_{t,\mathcal{H}}} \sup_{0<y<\frac{1}{|x|}}|x|\left(\frac1{|x|}-y\right) + C\|f\|_{\Lambda^{1/2,1}_{t,\mathcal{H}}}  \\
	 \le\|f\|_{\Lambda^{1/2,1}_{t,\mathcal{H}}}  |x|\sup_{0<y<\frac{1}{|x|}} \int_y^{\frac{1}{|x|}}  \int_{z_1}^\frac1{|x|} z_2^{-1}dz_2  dz_1+ C\|f\|_{\Lambda^{1/2,1}_{t,\mathcal{H}}} . 	\end{multline*}
	Since for every $0<y<\frac{1}{|x|}$ we have
	\begin{align*}  |x|\int_y^{\frac{1}{|x|}}  \int_{z_1}^\frac1{|x|} z_2^{-1}dz_2  dz_1
	&=|x|\int_y^{\frac{1}{|x|}} \left( \log\left(\frac1{|x|}\right) - \log z_1\right)\, dz_1  \\&=|x|\left[
	\log\left(\frac1{|x|}\right) \left(\frac1{|x|}-y\right) -  \left( \frac1{|x|}\log\frac1{|x|}- \frac1{|x|} -
	y\log y + y \right) \right]\\&= |x|y\log (|x|y)+ |x|\left(\frac1{|x|}-y\right)\le C,
	\end{align*}	
	we conclude that $|x| |f(t,x)| \le C\|f\|_{\Lambda^{1/2,1}_{t,\mathcal{H}}}$.
	
	For the cases in which $\alpha$ is an integer bigger that 1, we have to write $\partial_{z_1} \mathcal{P}_{z_1}f$ in terms of the integral of the derivative of order $k$, where $k=[\alpha]+1$, and proceed analogously. We leave  the details to the interested reader.
\end{proof}

 \subsection{Proof of Theorem \ref{teo151}.}

 \begin{proof} {\bf Proof of epigraph (1) in Theorem \ref{teo151}.}
 	Let $k=[\alpha]+1$. Since 
 	\newline $\displaystyle  \int_{\mathbb{R}^{n+1}} \partial_y^k \mathcal{P}_y(\tau,x,z)f(t-\tau,x+z)d\tau dz=\int_{\mathbb{R}^{n+1}} \partial_y^k \mathcal{P}_y(\tau,x,-z)f(t-\tau,x-z)d\tau dz,$ 
	\newline and  $\displaystyle \int_0^\infty\partial_y^k\left(\frac{ye^{-\frac{y^2}{4\tau}}}{\tau^{3/2}} \right)d\tau=0$ we have
 	\begin{align}\label{ageneralizar}
\nonumber 	\partial_y^k \mathcal{P}_yf(t,x) &= \frac{1}{2}\int_{\mathbb{R}^{n+1}}\partial_y^k \mathcal{P}_y(\tau,x,z)(f(t-\tau,x-z)+f(t-\tau,x+z)-2f(t,x)) d\tau dz 
 	\\  & \quad \quad+\frac{1}{2}\int_{\mathbb{R}^{n+1}}\Big(\partial_y^k \mathcal{P}_y(\tau,x,z)-\partial_y^k \mathcal{P}_y(\tau,x,-z)\Big) f(t-\tau,x-z)d\tau dz \\ \nonumber & \quad \quad  +\frac{f(t,x)}{2\sqrt{\pi}} \int_{0}^\infty\partial_y^k\left(\frac{ye^{-\frac{y^2}{4\tau}}}{\tau^{3/2}} \right)\left(e^{-\tau{\LL}}1(t,x)-1\right)d\tau \\ \nonumber &= I_1+I_2+I_3.
 	\end{align}
 	By Lemma \ref{lemma5B},  
 	$\displaystyle 
 	|I_1|\le C\int_{\mathbb{R}^n}\int_0^\infty\frac{e^{-\frac{y^2+|z|^2}{c\tau}}(\tau^{1/2}+|z|)^\alpha}{\tau^{\frac{n+k}{2}}}\frac{d\tau}{\tau}dz\le C y^{\alpha-k}.
 	$
 	For $I_3$ {we use Proposition \ref{taman} and the proof of Lemma \ref{'lemma4'} (3)  to get}
 	\begin{align*}
 	|y^kI_3|=\bigg|\frac{f(t,x)}{2\sqrt{\pi}} \int_{0}^\infty y^k\partial_y^k\left(\frac{ye^{-\frac{y^2}{4\tau}}}{\tau^{3/2}} \right)\left(e^{-\tau{\LL}}1(t,x)-1\right)d\tau \bigg|\leq  C [f]_{M^\alpha}y^{\alpha}.
 	\end{align*}
 	Regarding $I_2,$ we have
 	\begin{align*}
 &	2I_2=\int_{\mathbb{R}^{n+1}}(\partial_y^k \mathcal{P}_y(\tau,x,z)-\partial_y^k \mathcal{P}_y(\tau,x,-z)) f(t-\tau,x-z) d\tau dz 
 	\\ &=\int_{\mathbb{R}^n}\int_0^\infty\partial_y^k\left(\frac{ye^{-\frac{y^2}{4\tau}}}{\tau^{3/2}}\right) \frac{e^{-\frac{|z|^2}{4}\coth\tau}}{2\sqrt{\pi}(2\pi\sinh(2\tau))^{n/2}}\left( e^{-\frac{|2x-z|^2}{4}\tanh \tau}- e^{-\frac{|2x+z|^2}{4}\tanh \tau}\right)\\ &  \times f(t-\tau,x-z) d\tau dz. 
 	\end{align*}
 	By the Mean Value Theorem applied to the function  $e^{-\frac{|2x-z|^2}{4}\tanh \tau}$   we get
 	\begin{align*}
 	|2I_2| &\le C\int_0^\infty\partial_y^k\left(\frac{ye^{-\frac{y^2}{4\tau}}}{\tau^{3/2}}\right) \int_{\mathbb{R}^n}
 	\frac{e^{-\frac{|z|^2}{4}\coth\tau}}{(\sinh(2\tau))^{n/2}}(\tanh\tau)^{1/2}|z||f(t-\tau,x-z)|dz  d\tau\\
 	&\underbrace{\leq}_{\substack{\frac{z\sqrt{\coth\tau}}{2}=w}}C\|f\|_\infty\int_0^\infty\partial_y^k\left(\frac{ye^{-\frac{y^2}{4\tau}}}{\tau^{3/2}}\right)  \int_{\R^n}\frac{e^{-|w|^2}|w|(\tanh\tau)^{1/2}}{(\sinh(2\tau))^{n/2}(\coth\tau)^{\frac{n+1}{2}}}dw d\tau\\ &\leq C_k\|f\|_\infty\int_0^\infty \frac{e^{-\frac{y^2}{c\tau}}}{\tau^{k/2+1}}\frac{(\tanh\tau)^{1/2}}{(\sinh(2\tau))^{n/2}(\coth\tau)^{\frac{n+1}{2}}} d\tau
 \\ & 	\le C_k \|f\|_\infty \int_0^\infty \frac{e^{-\frac{y^2}{c\tau}}}{\tau^{k/2}}\tau^{\alpha/2}\frac{d\tau}{\tau} \le C_k y^{-k+\alpha}.
 	\end{align*}
 	We conclude that $f\in	\Lambda^\alpha_{\mathcal{L}}$.

 	For  the converse. 
 	If  $f\in	\Lambda^\alpha_{\mathcal{L}}$ with $\alpha<1$.  the result is a consequence of Theorem \ref{caracHolder2}.  
 	If $\alpha \ge 1$, by Theorem \ref{caracHolder2},   $f\in\Lambda_{\mathcal{L}}^{\alpha'}=C_{t,\mathcal{H}}^{\alpha'/2,\alpha'}$ for some $\alpha'<1$, then $\|y\partial_y\mathcal{P}_yf\|_{L^\infty(\mathbb{R}^{n+1})}\rightarrow 0$, as $y\to 0^+$.
 	On the other hand, by  the proof of Proposition \ref{convergencia2}  we know that $\|\mathcal{P}_yf-f\|_{L^\infty(\mathbb{R}^{n+1})}\rightarrow 0$, as $y\to 0^+$ Hence we have
 	$$
 	f(t,x)=\int_0^y y'\frac{\partial^2\mathcal{P}_{y'}f(t,x)}{(\partial y')^2}dy'-y\partial_y\mathcal{P}_yf(t,x)+\mathcal{P}_yf(t,x).
 	$$
	We only do computations for $g(t,x)=\mathcal{P}_y f(t,x)$. For the other cases we have to follow the same path. By using Lemma \ref{est2} we have, for $y=\tau^{1/2}+|z|$, 
 	\begin{small}
 		\begin{align*}
 		|g(t-\tau,x+z)+&g(t-\tau,x-z)-2g(t,x)|\\ &\le |\left[\nabla_x g(t-\tau, x+\theta z)-\nabla_x g(t-\tau, x-\lambda z)\right]| |z|+2|\partial_t g(t- \eta\tau,x)|\tau\\
 		&\leq\left|  D_x^2g(t-\tau,x+\nu z)\right|(\theta+\lambda)|z|^2+2|\partial_t\mathcal{P}_yf(t-\eta\tau,x)|\tau\\
 		&\le C\|f\|_{	\Lambda^\alpha_{\mathcal{L}} }(\tau^{1/2}+|z|)^{-2+\alpha}(|z|^2+\tau)\le C\|f\|_{	\Lambda^\alpha_{\mathcal{L}} }(\tau^{1/2}+|z|)^{\alpha},
 		\end{align*}
 	\end{small}
 	where $0<\theta,\lambda<1$, $-1<\nu<1$.
 
 	The fact that $(1+|x|)^\alpha f\in L^\infty(\R^{n+1})$ follows from Proposition \ref{taman}.
 	
 \end{proof}

For the proof of epigraph (2) in Theorem \ref{teo151}, we shall prove 
the following theorem. 
\begin{theorem}\label{teo15}
 Suppose that $\alpha>2$. Then $f\in	\Lambda^\alpha_{\mathcal{L}}$ if and only if
		$$
		\partial_{x_i}f, x_if\in 	\Lambda^{\alpha-1}_{\mathcal{L}}, \:\;i=1,\dots,n,
		\quad \hbox{  and  } \quad 
		{\partial_t}f \in 	\Lambda^{\alpha-2}_{\mathcal{L}}.
		$$
		In this case the following equivalence holds
		$$
		\|f\|_{\Lambda^{\alpha-2}_{\mathcal{H}}}\sim  \sum_{i=1}^n\left(\|{\partial}_{x_i}f\|_{	\Lambda^{\alpha-1}_{\mathcal{L}}} +\|x_if\|_{	\Lambda^{\alpha-1}_{\mathcal{L}}}\right)+ \|{\partial_t}f \|_{	\Lambda^{\alpha-2}_{\mathcal{L}}}.
		$$
\end{theorem} 
For the reader's convenience, the proof of this Theorem \ref{teo15} will be divide in several steps.

 \begin{proposition}\label{derivent}
 	Suppose that $f\in	\Lambda^\alpha_{\mathcal{L}}$ with $\alpha>2$. Then,
 	$\displaystyle
 	\partial_tf\in 	\Lambda^{\alpha-2}_{\mathcal{L}}.
 	$
 \end{proposition}

\begin{proof} Let $2 < \alpha < 3$, by Lemma \ref{est2} we have 
\begin{equation}\label{crucialt} \|\partial_y \partial_t \mathcal{P}_yf \|_\infty \le  \|f\|_{	\Lambda^\alpha_{\mathcal{L}}}y^{-3+\alpha}
\end{equation} If  $y<1$ we have  $
\partial_{t} \mathcal{P}_{y} f = \int_y^1  \partial_z \partial_t \mathcal{P}_zf dz + \partial_{t} \mathcal{P}_{y} f \Big|_{y=1}, $ this implies $\partial_{t} \mathcal{P}_{y} f$ is in $L^\infty(\mathbb{R}^{n+1}) $ uniformly on $y$. Moreover  since 
 $| \partial_{t} \mathcal{P}_{y'} f-  \partial_{t} \mathcal{P}_{y} f | \le \|f\|_{	\Lambda^\alpha_{\mathcal{L}}} \int_{y}^{y'} z^{-3+\alpha}dz \to 0$ as $(y',y) \to 0$, then   $ \partial_{t} \mathcal{P}_{y} f$ converges uniformly when $y\to 0. $  As $\mathcal{P}_y f$ converges uniformly to $f$ when $y\to 0$, we conclude that $\partial_{t}  f$ exists, it  is  the uniform limit of 
 $\partial_{t} \mathcal{P}_{y} f= \mathcal{P}_{y} \partial_t f.$ Hence 
 $\partial_y\mathcal{P}_{y} \partial_{t} f = \partial_y \partial_t\mathcal{P}_{y}  f. $ The last identity together with inequality \eqref{crucialt} implies  $\partial_t f \in 	\Lambda^{\alpha-2}_{\mathcal{L}}.$

If $\alpha \ge 3$,  by Corollary \ref{encajadas2}, the function $f\in\Lambda^{\beta}_{\LL}$ for some $\beta <1.$ Hence by the thoughts developed before, $\partial_t f $ exists and 
$\partial_{t} \mathcal{P}_{y} f= \mathcal{P}_{y} \partial_t f$. The  proof follows the lines of  the case $2<\alpha < 3$.	

	\end{proof}

 \begin{proposition}\label{derivadas2}
 	Suppose that $f\in\Lambda^{ \alpha}_{\LL}$ with $\alpha>1$. Then,
 	$ \displaystyle
 	\frac{\partial  f}{\partial x_i}\in 	\Lambda^{\alpha-1}_{\mathcal{L}}, \:\;i=1,\dots,n.
 	$
 \end{proposition}
 
 \begin{proof}
 	
 	Let $1< \alpha <3$.  By  Lemma \ref{est2}   we have $\displaystyle\bigg\| \frac{\partial^{3} \mathcal{P}_{y}f}{\partial y^2\partial x_i}\bigg\|_\infty\leq C \|f\|_{\Lambda^{\alpha}_{\LL}} y^{-3+\alpha}.$ 
 	For $y<1$, an integration gives 
 	$$\Big| \frac{ \partial^2 \mathcal{P}_y f(t,x) }{\partial y \partial x_i}\Big| \le  C \|f\|_{\Lambda^{\alpha}_{\LL}}y^{-2+\alpha} +C \Big\| \frac{ \partial^2 \mathcal{P}_y f }{\partial y \partial x_i} \Big|_{y=1} \Big\|_\infty.$$
 	 We can proceed as in the proof Proposition \ref{derivent} and we get that  $\partial_{x_i}f$ does exist and  $\left\| \frac{\partial f}{\partial x_i}\right\|_\infty\le C .$
	 To prove that $\frac{\partial f}{\partial_{x_i}}\in 	\Lambda^{\alpha-1}_{\mathcal{L}},$ we shall see that $\| \partial^2_y \mathcal{P}_y(\partial_{x_i}f) \|_\infty\le C\|f\|_{\Lambda^{\alpha}_{\LL}} y^{-3+\alpha}$. Observe that
 	$$                                                                                       
 	\partial^2_y \mathcal{P}_y(\partial_{x_i}f)(t,x)=\partial_{x_i}\partial^2_y \mathcal{P}_yf(t,x)-\int_{\R^{n+1}}\partial_{x_i}\partial^2_y\mathcal{P}_y(\tau,x,z)f(t-\tau,x-z)d\tau dz =I+II.
 	$$
 	By Lemma \ref{est2} we have that $|I|\le  C \|f\|_{\Lambda^{\alpha}_{\LL}}y^{-3+\alpha}=C \|f\|_{\Lambda^{\alpha}_{\LL}}y^{-2+(\alpha-1)}$.  As $f\in \Lambda^{\alpha}_{\LL}$, $1<\alpha <3$, by Proposition \ref{taman} we know that  $|x|f\in L^\infty(\R^{n+1})$.
 	Hence, by Lemma \ref{lemma5B} (iii) we get  that $|II|\le C \|f\|_{\Lambda^{\alpha}_{\LL}}y ^{-3+\alpha}= C\|f\|_{\Lambda^{\alpha}_{\LL}} y^{-2+(\alpha-1)}.$
 	
 	Suppose now $3\le \alpha <5$.  By Corollary \ref{encajadas2}   $f\in \Lambda^{\beta}_{\LL}$ for all $\beta < 3$. Then, the result just proved says that  $\frac{\partial f}{\partial{x_i}}\in \Lambda^{\gamma}_{\LL}$, for all $ \gamma <2$ and 
	$\frac{\partial^2 f}{\partial{x^2_i}}\in \Lambda^{\delta}_{\LL}$, for all $ \delta <1$.  We shall see that $\|\partial_y^{4}\mathcal{P}_y(\partial_{x_i}f)\|_\infty\le C\|f\|_{\Lambda^{\alpha}_{\LL}}y^{-4+(\alpha-1)}.$  As $\mathcal{P}_y(\partial_{x_i}f)$ satisfies \eqref{EqPoisson}, it is enough to prove that $\|\partial_y^{2}( -\sum_{j=1}^n \partial^2_{x_j}+|x|^2 +\partial_t)\mathcal{P}_y(\partial_{x_i}f)\|_\infty\le C\|f\|_{\Lambda^{\alpha}_{\LL}}y^{-5+\alpha}.$ 
 	
 Observe that  
 	\begin{align*}
 	\partial_y^{2}\partial_{x_j}^2&\mathcal{P}_y(\partial_{x_i}f)(t,x)=\partial_y^{2}\partial_{x_j}^2\partial_{x_i}\mathcal{P}_yf(t,x)-\partial_{x_j}^2\int_{\R^{n+1}}\partial_y^{2}\partial_{x_i}\mathcal{P}_y(\tau,x,z)f(t-\tau,x-z)d\tau dz \\
 	&=\partial_y^{2}\partial_{x_j}^2\partial_{x_i}\mathcal{P}_yf(t,x)-\int_{\R^{n+1}}\partial_y^{2}\partial_{x_j}^2\partial_{x_i}\mathcal{P}_y(\tau,x,z)f(t-\tau,x-z)d\tau dz \\&-2\int_{\R^{n+1}}\partial_y^{2}\partial_{x_i}\partial_{x_j}\mathcal{P}_y(\tau,x,z)\partial_{x_j}f(t-\tau,x-z)d\tau dz 
 	\\&-\int_{\R^{n+1}}\partial_y^{2}\partial_{x_i}\mathcal{P}_y(\tau,x,z)\partial_{x_j}^2f(t-\tau,x-z)d\tau dz.
 	\end{align*}
 	The first summand is bounded by $C\|f\|_{\Lambda^{\alpha}_{\LL}}y^{-5+\alpha}$ because of Lemma \ref{est2}. As $f$ and $\partial_{x_i}f$ are bounded functions, by using Lemma \ref{lemma5B} (ii) we get the desired boundedness for the second and third summand. Finally Lemma \ref{lemma5B} (iii)  says that the forth summand is bounded by $C y^{-(1-\nu+s)}$, where $\nu <1$ and  $s>0$, then by choosing $\nu$ and $s$ with $s-\nu = 4 - \alpha$ we get the estimate.
	
 On the other hand,  by using Lemma \ref{lemma5B} (ii) and (iii) together with the facts that $f, {\partial_{x_i}f}\in L^\infty(\R^{n+1})$ and $\frac{\partial^2 f}{\partial{x^2_i}}	 \in \Lambda^{\beta-2}_{\LL}$, we get the desired estimate in this case.

  To prove that $\||\cdot|^{2}\partial^{2}_y\mathcal{P}_y(\partial_{x_i}f)\|_\infty\le C\|f\|_{\Lambda^{\alpha}_{\LL}}y^{-5+\alpha}$,  we write
 	$$
 	|x|^2\partial_y^{2}\mathcal{P}_y(\partial_{x_i}f)(t,x)=|x|^2\partial_{x_i}\partial_y^{2}\mathcal{P}_yf(t,x)-|x|^2\int_{\R^{n+1}}\partial_{x_i}\partial_y^{2} \mathcal{P}_y(\tau,x,z)f(t-\tau,x-z)d\tau dz
 	$$
 	By Lemma \ref{est2} we know that the first summand is bounded by $C\|f\|_{\Lambda^{\alpha}_{\LL}}y^{-5+\alpha}$. For the second summand we have
\begin{multline*}
 	|x|^2\int_{\R^{n+1}}|\partial_{x_i}\partial_y^{2} \mathcal{P}_y(\tau, x,z)f(t-\tau, x-z)|d\tau dz \\ \le C \int_{\R^{n+1}}|\partial_{x_i}\partial_y^{2} \mathcal{P}_y(\tau,x,z)|(|x-z|^2+|z|^2)|f(t-\tau, x-z)|d\tau dz,
\end{multline*}
 	and by Lemma \ref{lemma5B} (iii) applied to $|x|^2f$ and Lemma \ref{lemma5B} (ii) we get the desired bound $C\|f\|_{\Lambda^{\alpha}_{\mathcal{L}}}y^{-5+\alpha}$.

 	To get the estimate for $\|\partial_y^2\partial_t\mathcal{P}_y(\partial_{x_i}f)\|_\infty$, we write
 	\begin{align*}
 	\partial_y^2\partial_t\mathcal{P}_y(\partial_{x_i}f)(t,x)&=\partial_y^2\partial_{x_i}\partial_t\mathcal{P}_yf(t,x)-\int_{\R^{n+1}}\partial_{x_i}\partial_y^2\mathcal{P}_y(\tau,x,z)\partial_tf(t-\tau,x-z)d\tau dz.
 	\end{align*}
   By Proposition \ref{derivent} we know that $\partial_tf\in  	\Lambda^{\alpha-2}_{\mathcal{L}}$, $1\le \alpha-2<3$. Hence, as $\partial_y^2\partial_{x_i}\partial_t\mathcal{P}_yf(t,x)=\partial_y^2\partial_{x_i}\mathcal{P}_y(\partial_tf)(t,x)$,  by applying Lemma \ref{est2} $(1)$ we get that the first summand is bounded by $C\|f\|_{\Lambda^{\alpha}_{\mathcal{L}}}y^{-5+\alpha}$, and by Lemma \ref{lemma5B} (iii) applied to $\partial_tf$ we get the same bound for the second summand.
 
 	The rest of the cases, $2m+1 \le \alpha < 2m+3,$ can be handled analogously by estimating the norms  $\|\partial_y^{2}(- \sum_j \partial^2_{x_j}+|x|^2+\partial_t )^m\mathcal{P}_y(\partial_{x_i}f)\|_\infty.$ We leave the details to the reader.
 \end{proof}
 
 \begin{proposition}\label{porx}
 Suppose that $f\in\Lambda^{ \alpha}_{\LL}$ with $\alpha>1$. Then,
 $ \displaystyle
  x_if\in 	\Lambda^{\alpha-1}_{\mathcal{L}}, \:\;i=1,\dots,n.
 $
 \end{proposition}

 \begin{proof}
 	
 	Consider the case $1<\alpha<2$. By Proposition  \ref{taman} we know that  ${x_i}f\in L^\infty(\R^{n+1})$.  In addition, we can write
 	$$
 	\partial_y\mathcal{P}_y(x_if)(t,x)=x_i \partial_y\mathcal{P}_yf(t,x)-\int_{\R^{n+1}}z_i\mathcal{P}_y(\tau,x,z)f(t-\tau,x-z)d\tau dz,
 	$$
	
 	and by using Lemma \ref{est2} for the first summand  and Lemma \ref{lemma5B} together with the boundedness of $f$ for the second summand, we get  that $\|	\partial_y\mathcal{P}_y(x_if) \|_\infty\le C \|f\|_{\Lambda^{\alpha}_{\mathcal{L}}} y^{-2+\alpha}$.

 	Let $2\le \alpha <3$.  We have to prove that $\|\partial_y^{2} \mathcal{P}_y({x_i}f)\|_\infty \le C\|f\|_{\Lambda^{\alpha}_{\mathcal{L}}}y^{-3+\alpha}$. As $\mathcal{P}_y(x_if)$ satisfies \eqref{EqPoisson} we have 
 	\begin{align*}
 	\left\| {\partial_y^2 \mathcal{P}_y({x_i}f)} \right\|_\infty &= \left\|\Big[\partial_t- \sum_{j=1}^n(\partial_{x_j}^2-|x|^2)\Big]\mathcal{P}_y(x_if)\right\|_\infty \\ &\le \|\partial_t\mathcal{P}_y(x_if)\|_\infty+ \sum_{j=1}^n\| \partial_{x_j}^2\mathcal{P}_y(x_if)\|_\infty+\||\cdot|^2 \mathcal{P}_y(x_if) \|_\infty.
 	\end{align*}
As  $\partial_tf$ is well defined  and bounded, see  Proposition \ref{derivent},   
\begin{align*}
\partial_t \mathcal{P}_y(x_if)(t,x)&=\int_{\R^{n+1}}\mathcal{P}_y(\tau,x,z)(x_i-z_i)\partial_tf(t-\tau, x-z)d\tau dz\\&=x_i\mathcal{P}_y(\partial_tf)(t,x)-\int_{\R^{n+1}}z_i\mathcal{P}_y(\tau,x,z)\partial_tf(t-\tau, x-z)d\tau dz.
\end{align*}

 Therefore, by using  Proposition \ref{derivent} and Lemma \ref{est2} $(1)$  for $\partial_tf$, we get that the first summand is bounded by $C\|\partial_tf\|_{\Lambda^{ \alpha-2}_{\LL}}y^{-1+(\alpha-2)}\le C\|f\|_{\Lambda^{\alpha}_{\mathcal{L}}}y^{-3+\alpha}$. For  the second summand we use that $\partial_tf\in L^\infty(\R^{n+1})$ and Lemma \ref{lemma5B} (ii).

 To get the bound for $\||\partial_{x_j}^2 \mathcal{P}_y(x_if) \|_\infty$, for $j=1,\dots,n$. We can write, for every $j\in\{1,\dots,n\}$,
 	\begin{multline*}
 	\partial_{x_j}^2\mathcal{P}_y(x_if)(t,x)=\partial_{x_j}^2\int_{\R^{n+1}}\mathcal{P}_y(\tau,x,z) x_if(t-\tau,x-z)d\tau dz\\ \quad \quad  -\partial_{x_j}^2\int_{\R^{n+1}}\mathcal{P}_y(\tau,x,z) z_if(t-\tau,x-z)d\tau dz\\
 	=x_i\partial_{x_j}^2\mathcal{P}_yf(t,x)+2 \delta_{i,j}\int_{\R^{n+1}}\partial_{x_j}(\mathcal{P}_y(\tau,x,z)f(t-\tau,x-z))d\tau dz \\  -\int_{\R^{n+1}}\partial_{x_j}^2(\mathcal{P}_y(\tau, x,z)z_i)f(t-\tau,x-z)d\tau dz  \\ -2\int_{\R^{n+1}}\partial_{z_i}(\partial_{x_j}\mathcal{P}_y(\tau,x,z)z_i)f(t-\tau,x-z)d\tau dz \\ -\int_{\R^{n+1}}\partial_{z_j}(\mathcal{P}_y(\tau,x,z)z_i)\partial_{x_j}f(t-\tau,x-z)d\tau dz,
 	\end{multline*}
 	where  $\delta_{i,j}=1$ if $i=j$ and $\delta_{i,j}=0$ if $i\neq j$. Observe that in the last summand we have used integration by parts.
 	As $f\in \Lambda^{\alpha}_{\LL}$, by  Lemma \ref{est2} $(1)$ we get that the first summand is bounded by  $C\|f\|_{\Lambda^{\alpha}_{\mathcal{L}}}y^{-3+\alpha}$.  For the rest of summands we can apply  Lemma \ref{lemma5B} (ii) since $f$ and $\partial_{x_i} f $ are bounded functions.
 	 	
In remains  the case $\||\cdot|^2 \mathcal{P}_y(x_if) \|_\infty.$ Observe that, 
 	\begin{multline*}
 	|x|^2|\mathcal{P}_y(x_if)(t,x)|  \le C \int_{\R^{n+1}}(|x-z|^2+|z|^2)|\mathcal{P}_y(\tau,x,z)||x_i-z_i||f(t-\tau,x-z)|d\tau dz \\ \le C \int_{\R^{n+1}}|x-z|^{3-\alpha}|\mathcal{P}_y(\tau,x,z)||x-z|^\alpha|f(t-\tau,x-z)|d\tau dz\\+ C \int_{\R^n}|z|^2|\mathcal{P}_y(\tau,x,z)||x-z||f(t-\tau,x-z)|d\tau dz\\
 	\le C(\||x|^\alpha f\|_\infty+\||x|f\|_\infty)\int_{\R^{n+1}}(|x|^{3-\alpha}+|z|^{3-\alpha}+|z|^2)|\mathcal{P}_y(\tau,x,z)|d\tau dz\le C\|f\|_{\Lambda^{\alpha}_{\LL}}y^{-3+\alpha}.
 	\end{multline*}
	In the last inequality  we have used Lemma \ref{lemma5B} (ii).
 It remains  	 	
 	 	For the cases $2m+1\le \alpha<2m+3$, with $m\ge 1$,  we get the result by following the same kind of reasonings, that is, by estimating the norms  $\|(- \sum_j \partial^2_{x_j}+|x|^2+\partial_t )^{m+1}\mathcal{P}_y({x_i}f)\|_\infty$. We leave  the details for the interested reader.

 \end{proof} 
 \begin{proposition}\label{bajada}  
 	Let $\alpha>2$ and $f\in L^\infty(\R^{n+1})$ and suppose that $\partial_{x_i}f,x_if\in 	\Lambda^{\alpha-1}_{\mathcal{L}}$,  $i=1,\dots,n$, and $\partial_tf\in 	\Lambda^{\alpha-2}_{\mathcal{L}}$. Then  $f\in 	\Lambda^{\alpha}_{\mathcal{L}}$. 
 \end{proposition}
 \begin{proof}

 Consider the case $2\le \alpha<4$. We want to see that $\|\partial_y^4\mathcal{P}_yf\|_\infty \le  C y^{-4+\alpha}$, and as $\mathcal{P}_yf$ satisfies \eqref{EqPoisson}, we have that
 	$\partial_y^4\mathcal{P}_yf(t,x)=\left(\partial_t- \sum_{j=1}^n\partial_{x_j}^2+ |x|^2\right)^2 \mathcal{P}_yf(t,x)$. Hence it is sufficient to prove that
 	\begin{enumerate}
 		\item [a)] $\|\partial_{x_j}^4\mathcal{P}_yf\|_\infty \le  C y^{-4+\alpha}$,
 		\item [b)]  $\|\partial_{x_j}^2(|x|^2\mathcal{P}_yf)\|_\infty \le  C y^{-4+\alpha}$,
 		\item [c)]  $\||x|^2\partial_{x_j}^2\mathcal{P}_yf\|_\infty \le  C y^{-4+\alpha}$, $j\in\{1,\dots,n\}$, and 
 		\item [d)]  $\||x|^4\mathcal{P}_yf\|_\infty \le  C y^{-4+\alpha}$.
 			\item [e)]  $\||\partial_t^2\mathcal{P}_yf\|_\infty \le  C y^{-4+\alpha}$.
 				\item [f)]  $\|\partial_t|x|^2\mathcal{P}_yf\|_\infty \le  C y^{-4+\alpha}$.
 					\item [g)] $\||\partial_t\partial_{x_j}^2\mathcal{P}_yf\|_\infty \le  C y^{-4+\alpha}$.
 			 	\end{enumerate}
Integration by parts gives 
 	\begin{align*}
 	\partial_{x_j}^4\mathcal{P}_yf(t,x) &=\partial_{x_j}^3\int_{\R^{n+1}}\mathcal{P}_y(\tau,x,z)\partial_{x_j}f(t-\tau,x-z)d\tau dz  \\ & + \int_{\R^{n+1}}\partial_{x_j}^3(\partial_{x_j}+ \partial_{z_j})\mathcal{P}_y(\tau,x,z)f(t-\tau,x-z)d\tau dz\\&+2\int_{\R^{n+1}}\partial_{x_j}^2(\partial_{x_j} +\partial_{z_j}) \mathcal{P}_y(\tau,x,z)\partial_{x_j}f(t-\tau,x-z)d\tau dz
 	\\&+\int_{\R^{n+1}}\partial_{x_j}(\partial_{x_j} + \partial_{z_j}) \mathcal{P}_y(\tau,x,z)\partial_{x_j}^2f(t-\tau,x-z)d\tau dz. 	\end{align*}
As $\partial_{x_j}f\in	\Lambda^{\alpha-1}_{\mathcal{L}}$, by Lemma \ref{est2} we get that the first summand is bounded by $C\|\partial_{x_j}f\|_{	\Lambda^{\alpha-1}_{\mathcal{L}}}y^{-4+\alpha}$. For the rest of the summands we apply Lemma  \ref{lemma5B} together of the boundedness of the functions $f,\partial_{x_j} f$ and $\partial^2_{x_j}f.$	
 	To prove b), we write 
 	\begin{align}\label{c)} 
 \partial_{x_j}^2&(|x|^2\mathcal{P}_yf)(t,x)=2\mathcal{P}_yf(t,x)+4\int_{\R^{n+1}}x_j\partial_{x_j}\mathcal{P}_y(\tau,x,z)f(t-\tau,x-z)d\tau dz \\  \nonumber &+4\int_{\R^{n+1}}\mathcal{P}_y(\tau,x,z)  x_j\partial_{x_j} f(t-\tau,x-z)d\tau dz \\ & \nonumber +\int_{\R^{n+1}}|x|\partial_{x_j}\mathcal{P}_y(\tau,x,z)|x|\partial_{x_j}f(t-\tau,x-z)d\tau dz \\  &+\int_{\R^{n+1}}|x|^2\partial^2_{x_j}\mathcal{P}_y(\tau,x,z)f(t-\tau,x-z)d\tau dz +|x|^2\partial_{x_j}\mathcal{P}_y(\partial_{x_j}f)(t,x).  \nonumber
 	\end{align}
 As the functions $f$ and $|x| \partial_{x_j} f$ are bounded, Lemma \ref{lemma5B} takes care of the first to forth summands.  	The bound of last summand in \eqref{c)} follows from the fact that $\partial_{x_j}f\in 	\Lambda^{\alpha-1}_{\mathcal{L}}$ and Lemma \ref{est2}.

   To see d), we use that $|x|^\alpha f\in L^\infty(\R^{n+1})$ and Lemma \ref{lemma5B} to get
 	\begin{align*}
 	|&|x|^4\mathcal{P}_yf(t,x)|\le C \int_{\R^{n+1}}\mathcal{P}_y(\tau,x,z)(|x-z|^{4-\alpha}|x-z|^\alpha+|z|^4)|f(t-\tau,x-z)|d\tau dz\\&\le C(\||x|^\alpha f\|_\infty +\|f\|_\infty)\int_{\R^{n+1}}\mathcal{P}_y(\tau,x,z)(|x|^{4-\alpha}+|z|^{4-\alpha}+|z|^4)d\tau dz\\ &\le  C(\||x|^\alpha f\|_\infty +\|f\|_\infty)y^{-4+\alpha}.
 	\end{align*}
 	Finally, for the estimates e)-g) observe that  	$$\|\partial_t^2\mathcal{P}_yf\|_\infty=\|\partial_t\mathcal{P}_y(\partial_tf)\|_\infty,  \|\partial_t\partial_{x_j}^2\mathcal{P}_yf\|_\infty=\|\partial_{x_j}^2\mathcal{P}_y(\partial_tf)\|_\infty,$$ and $\||x|^2\partial_t\mathcal{P}_yf\|_\infty=\||x|^2\mathcal{P}_y(\partial_tf)\|_\infty$.
 	Hence, by using that  $\partial_tf\in 	\Lambda^{\alpha-2}_{\mathcal{L}}$ and Lemma \ref{est2} $(1)$ we get the result.

 	For the rest of the  values of $\alpha$ we proceed analogously. We leave the details for the interested reader.
	This is the end of the proof of Propostion \ref{bajada}.
	\end{proof}
Propositions \ref{derivent}, \ref{derivadas2}, \ref{porx}, \ref{bajada} show the validity of Theorem  \ref{teo15}. Therefore we have proved Theorem \ref{teo151}, epigraph (2).

The proof of Theorem \ref{teo16} is now complete. As a consequence of it we get the following characterization of the spaces of Krylov's type introduced in Definition \ref{ParabolicHermite}.

\begin{theorem} Let $0<\alpha <3,$ $\alpha$ not an integer. 
  Then $$ C_{t,\mathcal{H}}^{\alpha/2,\alpha} =  {	\Lambda^\alpha_{\mathcal{L}}} ,$$ with equivalence of norms.

\end{theorem}

\begin{proof}
	The case $0<\alpha<1$ was proved in Theorem \ref{caracHolder2}.
	Consider $1<\alpha<2$. Suppose that $f\in	\Lambda^\alpha_{\mathcal{L}}$. By Theorem \ref{teo151}(1) we know that \eqref{eq21} holds, an by taking $z=0$ in this inequality we get that $f(\cdot,x)\in C^{\alpha/2}(\R)$ uniformly on $x$. In addition, by Propositions \ref{derivadas2} and \ref{porx} we  have that $(\partial_{x_i}\pm x_i)f\in \Lambda^{\alpha-1}_{\mathcal{L}}=C^{\frac{\alpha-1}{2},\alpha-1}_{t,\mathcal{H}}$. Thus, we get that $f\in C^{\frac{\alpha}{2},\alpha}_{t,\mathcal{H}}$.
Conversely, suppose that $f\in C^{\frac{\alpha}{2},\alpha}_{t,\mathcal{H}}$. Then, we have that $(\partial_{x_i}\pm x_i)f(t,\cdot)\in C^{\alpha-1}_{\mathcal{H}}$
uniformly on $t$ and $f(\cdot,x)\in C^{\alpha/2}(\R)$ uniformly on $x$. Hence, $(1+|x|)^\alpha f\in L^{\infty}(\R^{n+1})$ and 
\begin{align*}
|f(t-\tau,x-z)&+f(t-\tau,x+z)-2f(t,x)|\\
&\le |f(t-\tau,x-z)+f(t-\tau,x+z)-2f(t-\tau,x)|+2|f(t-\tau,x)-f(t,x)|\\&\le C|\nabla_xf(t-\tau, x+\theta z)-\nabla_xf(t-\tau, x-\lambda z)||z|+C\tau^{\alpha/2}\\
&\le C|\theta+\lambda|^{\alpha-1}|z|^{\alpha-1}|z|+C \tau^{\alpha/2}\le C(\tau^{1/2}+|z|)^{\alpha}.
\end{align*}
By Theorem \ref{teo151} (1)  we conclude that $f\in	\Lambda^\alpha_{\mathcal{L}}$.  The case $2< \alpha$ is a Corollary of Theorem \ref{teo15}.
	\end{proof}

\subsection{Elliptic Hermite setting} \label{EllipticSeting3}	

\

 Again as in the case of subsections \ref{EllipticSeting} and \ref{EllipticSeting2} we handled the functions $g(x)$ and $f(t,x) = g(x).$ The considerations made in that Remarks, together with  Theorems \ref{teo151} and \ref{teo15} give the proof of Theorem \ref{teo16}.
	
	Moreover the following Theorem is also true.
	\begin{theorem}
			 If $\alpha>0$ is not an integer,  we have $C^{\alpha}_{\mathcal{H}}=\Lambda^\alpha_{\mathcal{H}}$.
	\end{theorem}

\begin{remark} There exists a function $g \in \Lambda^1_{\HH}(\R)$, but so that  $\sup_{\{x: x\in [0,1], z\in [0,1]\}} |g(x+z)-g(x)| \le C z$ fails for all $C$. 

Consider the  functions $h$ and $\varphi$ as follows.  $h(x)= \sum_{k=1}^\infty 2^{-k} \cos^{2\pi  2^k x}$ and $\varphi$ is a positive differentiable  function, with continuous derivative, such that $\varphi(x) = 1$ when $ x\in [-3,3]$, and  for any $x$ there exist a constant $C$ with $(1+|x|) \varphi(x) \le C$ and $|\varphi' (x) | \le C.$ It is clear that $|h(x)| \le 1$, moreover  it can be checked, see \cite[Theorem 4.9]{Zygmund}, that 
$\|h(x+z)+h(x-z)-2h(x) \|_\infty \le A |z|.$ 

Now we choose the function $g(x) = h(x) \varphi(x)$, then by the properties of $h$ and $\varphi$ we have $|(1+|x|)g(x)| \le C$.  On the other hand by  the Mean Value Theorem we have 
\begin{align*}
\Big| g(x+z)&+g(x-z) -2g(x)\Big| \le  \Big| \,(h(x+z)+h(x-z)-2h(x) )\, \varphi(x+z)\Big| \\ &+ \Big|h(x-z)\,(\varphi(x-z)- \varphi(x+z))\Big|+ 2\Big| h(x)\, ( \varphi(x+z)-\varphi(x))\,\Big| \le C|z|.
\end{align*}
Now assume that $g$ satisfies $|g(x+z)-g(x)|\le C |z|$. Hence for $x,z\in [0,1]$ we would have $|h(x+z)-h(x)| \le C |z|$. But it is well know  that Weierstrass function doesn't satisfy Lipschitz condition, see \cite[Theorem 4.9]{Zygmund}.

\end{remark}

\section{Schauder and H\"older estimates}\label{fractional}

\begin{lemma}\label{Lpotencias} Let  $\alpha, \beta$ positive real numbers.
\begin{itemize}
\item[(a)] Let $0< 2 \beta < \alpha $ and $f \in \Lambda_{\mathcal{L}}^\alpha$ then we have $\LL^\beta f(t,x) \le C < \infty, \, (t,x) \in \mathbb{R}^{n+1}.$
\item[(b)]   For every $\beta>0$ and $f\in L^\infty(\R^{n+1})$ we have $\LL^{-\beta} f(t,x) \le C < \infty$, for all $(t,x) \in \mathbb{R}^{n+1}.$
\end{itemize}
\end{lemma}

\begin{proof} It suffices to consider the case  $2\beta <\alpha <[2\beta]+1=\ell.$   Then
\begin{align*}
\| (\mathcal{P}_\nu &f(t,x)-f(t,x))^{[2\beta]+1}\|_{L^\infty(\R^{n+1}) }=\left\| \int_0^\nu  \underbrace{\dots}_{\substack{\ell}} \int_0^\nu \partial_{y_1}\dots \partial_{y_\ell} \mathcal{P}_{y_1+\dots y_\ell}f(t,x) dy_\ell\dots dy_1
\right\|_\infty \\&\le C  \int_0^\nu  \underbrace{\dots}_{\substack{\ell}} \int_0^\nu (y_1+\dots y_\ell)^{-\ell+\alpha} dy_\ell  \dots dy_1
\le  C \nu^\alpha. \end{align*}
 Then, as $0< 2\beta< \alpha,$  and $f, P_\nu f \in L^\infty(\mathbb{R}^{n+1})$ we have
\begin{align}\label{eq31}
\LL^\beta f(t,x) & \le 
	  c_\beta \int_0^1 \frac{\nu^{\alpha}}{\nu^{1+2\beta}} d\nu + \int_1^\infty   \frac{1}{\nu^{1+2\beta}}  \le C < \infty.
	\end{align}

To prove $(b)$ we use the boundedness of $f$ for $\nu <1$ and Lemma \ref{lemma5B} (ii), with  $s_\beta>2\beta$ when $\nu>1$. Thus,
$$\LL^{-\beta} f(t,x)=\frac{1}{\Gamma(2\beta)}\int_0^\infty P_\nu f(t,x)\frac{d\nu}{\nu^{1-2\beta}}\le C_\beta\|f\|_\infty\left(\int_0^1\frac{d\nu}{\nu^{1-2\beta}}+\int_1^\infty \frac{d\nu}{\nu^{1+s_\beta-2\beta}} \right)\le C_\beta.
$$

\end{proof}

{\it  Proof of Theorem \ref{teoHolder}}.
Let $m= [\alpha-2\beta]+1$ and $\ell=[2\beta]+1$. Then, $m+\ell=[\alpha-2\beta]+1+[2\beta]+1 >\alpha-2\beta+2\beta=\alpha$, as $m+\ell\in \mathbb{N}$ we get $m+\ell \ge[\alpha]+1.$

Previous Lemma \ref{Lpotencias} and Fubini's Theorem allow us to write
\begin{align*}
\Big|\partial_y^m\mathcal{P}(\LL^\beta f) \Big|&= \Big|c_\beta \int_0^\infty \partial_y^m\mathcal{P}_y \Big( \int_0^\nu \underbrace{\dots}_{\substack{\ell=[2\beta]+1}} \int_0^\nu \partial_w^\ell\mathcal{P}_{w} |_{w= s_1+\dots+s_\ell} ds_1\dots ds_\ell \Big) \frac{d\nu}{\nu^{1+2\beta}} \Big|\\ &=
\Big|c_\beta \int_0^\infty \Big( \int_0^\nu \underbrace{\dots}_{\substack{\ell=[2\beta]+1}} \int_0^\nu \partial_w^{m+\ell}\mathcal{P}_{w} |_{w=y+ s_1+\dots+s_\ell} ds_1\dots ds_\ell \Big) \frac{d\nu}{\nu^{1+2\beta}}\Big|
\\ &\le
C_\beta \int_0^\infty \Big( \int_0^\nu \underbrace{\dots}_{\substack{\ell=[2\beta]+1}} \int_0^\nu  (y+s_1+\dots s_\ell)^{-(m+\ell) +\alpha} ds_1\dots ds_\ell \Big) \frac{d\nu}{\nu^{1+2\beta}} \\&=
C_\beta \int_0^y ( \dots )  \frac{d\nu}{\nu^{1+2\beta}}  + C_\beta \int_y^\infty  (\dots ) \frac{d\nu}{\nu^{1+2\beta}} =I +II,
\end{align*}
where in the last inequality we have used that  $m+\ell \ge [\alpha]+1 >\alpha$. Now we shall estimate $I $ and $II$.
\begin{align*} |I | &\le 
C_\beta y^{-m+\alpha} \int_0^y  \int_0^{\nu/y} \underbrace{\dots}_{\substack{\ell=[2\beta]+1}} \int_0^{\nu/y} (1+s_1+\dots s_\ell)^{-(m+\ell) +\alpha} ds_1\dots ds_\ell \frac{d\nu}{\nu^{1+2\beta}}  \\&\le 
C_\beta y^{-m+\alpha} \int_0^y   \Big(\frac{\nu}{y}\Big) ^\ell \frac{d\nu}{\nu^{1+2\beta}}  \le C_\beta y^{-m+\alpha-\ell} \int_0^y  \frac{d\nu}{\nu^{1+2\beta-\ell}}
\le C_\beta y^{-m+\alpha-2\beta}.
\end{align*}
Notice that in the last inequality we have used that $1+2\beta-\ell= 2\beta-[2\beta] <1.$ On the other hand,
\begin{align*}
|II|  &\le c_\beta \int_y^\infty  \Big(  (y+\nu)^{-m+\alpha} + y^{-m+\alpha} \Big) \frac{d\nu}{\nu^{1+2\beta}}.
\end{align*}
If $-m+\alpha \le 0$ we have $\displaystyle |II|  \le C  \int_y^\infty  y^{-m+\alpha}  \frac{d\nu}{\nu^{1+2\beta}} = Cy^{-m+\alpha-2\beta}$. While in the case $-m+\alpha > 0$, as $m -\alpha+2\beta+1  = [\alpha-2\beta] +1-\alpha+2\beta+1 >1$,  we get  $\displaystyle |II|  \le C  \int_y^\infty  \nu ^{-m+\alpha}  \frac{d\nu}{\nu^{1+2\beta}}  \le Cy^{-m+\alpha-2\beta}$.

\edproof

{\it Proof of Theorem \ref{teoSchau}.} 
Let $\ell = [\alpha+2\beta]+1 > [\alpha]+1 > \alpha$. Fubini Theorem together with Lemma \ref{Lpotencias} allow to get 
\begin{align*}\|\partial^\ell_y P_y (\mathcal{L}^{-\beta}f)(t,x&)\|_{L^\infty(\R^{n+1})} = \left\|\int_0^\infty \partial_y^\ell P_y P_\nu f(t,x) \frac{d\nu}{\nu^{1-2\beta}}\right\|_\infty \\ &\le
C\int_0^\infty (y+\nu)^{-\ell+\alpha}\frac{d\nu}{\nu^{1-2\beta}} \le C y^{-\ell+\alpha-2\beta}.  \end{align*}

For (b) we apply Lemma \ref{lemma5B} (ii), then for $\ell= [2\beta]+1$ we have  $|\partial_y^\ell P_y P_\nu f(t,x)| \le C \frac{\|f\|_{\infty}}{y^{\ell}}$. Then we can proceed as before.
\edproof

\vspace{0.5 cm}

{\it Proof of Theorem \ref{multiplicador}.} 
Let  $f\in L^\infty(\mathbb{R}^{n+1})$, by using Lemma \ref{lemma5B} (i) and (ii),  we have  \newline $ \displaystyle \Big| \int_0^\infty e^{-s\LL^{1/2}}f(t,x) a(s) ds \Big| \le C\int_0^\infty  \min(1, s^{-2}) ds \le C$. Moreover if  $f\in \Lambda^\alpha_\LL (\R^{n+1}), \, \alpha>0 $ and $\ell= [\alpha+1] +1 >\alpha+1 $, by Fubini's Theorem  we have 
\begin{align*}
\Big| \partial^\ell_y \mathcal{P}_y &\Big( \int_0^\infty \mathcal{P}_s f(t,x) a(s) ds \Big) \, \Big|=  \Big| \int_0^\infty \partial^\ell_w \mathcal{P}_w  f(t,x)\Big|_{w=y+s} a(s) ds \Big| \\&  \le C \Big| \int_0^\infty (y+s)^{-\ell +\alpha} ds \Big| \le  C y^{-\ell+\alpha+1}. 
\end{align*}
We have proved that the operator $f \longrightarrow \int_0^\infty e^{-s \LL^{1/2}} f a(s) ds$ maps $\Lambda^\alpha_{\LL} (\R^{n+1})$ into  $\Lambda^{\alpha+1}_{\LL} (\R^{n+1})$. Then Theorem \ref{teoHolder} gives the result.
\edproof

Finally the proof of Theorem \ref{Rieszm} is a direct consequence of Theorem \ref{teoSchau} and \ref{teo151}.
\subsection{Elliptic Hermite setting} \label{EllipticSeting4}

\

As we did in the previous Sections, we consider  $g(x)$ and  
$f(t,x) = g(x)$, then it can be easily checked that   $\HH^{\pm \beta} g(x) = \LL^{\pm \beta} f(t,x)$ and 
$m(\HH) = m(\LL)$. Hence Remarks \ref{EllipticSeting}, \ref{EllipticSeting2} and \ref{EllipticSeting2} show the Hermite's  version of Theorems \ref{teoHolder}, \ref{teoSchau}, \ref{multiplicador} and \ref{Rieszm}.

\section{Maximum and comparison principles.}\label{maximum}

\it{Proof of Theorem \ref{Maximump}. } Observe that $c_\beta >0$ for $[2\beta]+1$ odd and $c_\beta<0$ for $[2\beta]+1$ even. On the other hand as the kernel $P_\nu(\tau,x,z)$ is always positive we have $P_\nu f(t,x) \ge 0, \,  t\le t_0.$ 
If $0<\beta<1/2$,   $\displaystyle\LL^\beta f(t_0,x_0)=\frac{1}{c_\beta}\int_0^\infty P_\nu f (t_0,x_0) \frac{d\nu}{\nu^{1+2\beta}}$, then $\LL^\beta f(t_0,x_0)\le 0$.
If $1/2\le \beta<1$, then $\displaystyle \LL^\beta f(t_0,x_0)=\frac{1}{c_\beta}\int_0^\infty(P_{2\nu} f (t_0,x_0)-2P_\nu f (t_0,x_0)) \frac{d\nu}{\nu^{1+2\beta}}$, as  $(P_{2\nu} f (t_0,x_0)-2P_\nu f (t_0,x_0))\le 0$, we obtain that $\LL^\beta f(t_0,x_0)\le 0$.

\edproof

\section{Computational results}\label{cuentas}

The following remark will be used systematically along this  manuscript.
\begin{remark}\label{'nota3'}
Let $\tau>0$.
\begin{itemize}
\item[(1)] If $\tau<1$, then $\sinh\tau\sim\tau$, $\cosh\tau\sim C$, $\coth\tau\sim\frac{1}{\tau}$ and $\tanh\tau\sim\tau$.
\item[(2)] If $\tau>1$, then $\sinh\tau\sim e^\tau$, $\cosh\tau\sim e^\tau$, $\coth\tau\sim C$ and $\tanh\tau\sim C$.
\item[(3)] Given $ n \in \mathbb{N}$, $\ell\in\N$ and  $\lambda\ge0$,  there exists a positive constant 
$C_{\ell,n,\lambda}$ such that \newline   $\frac{1}{(\coth \tau)^{\ell}(\sinh \tau)^{n}} = \frac{(\tanh \tau)^{\ell}}{(\sinh \tau)^{n}} \le C_{\ell,n,\lambda} \min (\tau^{-n+\ell}, e^{-c\tau}) \le  C_{\ell,n,\lambda}
\tau^{-n+\ell-\lambda}.$
\item[(4)]  Let  $z\ge 0$ and $\alpha \ge 0$ there exist a constant $C_\alpha >0$ such that  $ z^\alpha e^{- z} \le C_\alpha e^{-z/2}.$
\end{itemize}
As usual by $A\sim B$ we mean there exist constants $C_1,C_2$ such that $C_1 A \le B \le C_2 A.$ 
\end{remark}

\begin{lemma}\label{'lemma4'}For each $x\in\R^n$ and $\tau>0$, we have: 
\begin{itemize}
\item[(1)] $\displaystyle
e^{-\tau {\LL}}1(t,x) =\frac{e^{-\frac{\tanh(2\tau)}{2}|x|^2}}{(\cosh(2\tau))^{n/2}}.
$

\item[(2)] $\displaystyle
|\partial_\tau e^{-\tau {\LL}}1(t,x) |
\leq C(\min\{\tau,1\}+|x|^2).$

\item[(3)] Given  $0 < \alpha <1$, there exists $C_\alpha >0$ such that
\begin{equation}\label{'eqP'}
\Big| e^{-y \sqrt{ {\LL}}}1(t,x)-1\Big|   \le C_\alpha (1+|x|)^\alpha y^\alpha .
\end{equation}
\end{itemize}

\end{lemma}

\begin{proof}
By using formula (\ref{MehlerK2}) we have
\begin{align*}
&(2\pi\sinh(2\tau))^{n/2}\displaystyle
e^{-\tau {\LL}}1(t,x)\\ &=\int_{\mathbb{R}^n}{\rm exp}\Big(-\frac{1}{4}\coth\tau(|x|^2+|z|^2-2xz)\Big){\rm exp}\Big(-\frac{1}{4}\tanh\tau(|x|^2+|z|^2+2xz)\Big)dz\\
&={\rm exp}({-\frac{1}{4}|x|^2(\coth\tau+ \tanh \tau)})\\ & \;  \times \int_{\mathbb{R}^n}{\rm exp}\Big(-\frac{1}{4}\left[\left(\sqrt{(\coth\tau+ \tanh \tau)}z-\frac{\coth\tau-\tanh\tau}{\sqrt{(\coth\tau+ \tanh \tau)}}x\right)^2-\frac{(\coth\tau-\tanh\tau)^2|x|^2}{(\coth\tau+ \tanh \tau)}\right]\Big)dz\\
&={\rm exp} \Big(-\frac{1}{4}|x|^2(\coth\tau+ \tanh \tau)\Big){\rm exp}\Big(\frac{1}{4}\frac{(\coth\tau-\tanh\tau)^2}{(\coth\tau+ \tanh \tau)}|x|^2\Big)\int_{\mathbb{R}^n}e^{-\frac{u^2}{4}}\frac{du}{(\coth\tau+ \tanh \tau)^{n/2}}\\
&={\rm exp}\Big(-\frac{1}{2}|x|^2\tanh(2\tau)\Big)\frac{(\sinh(2\tau))^{n/2}}{(2\cosh(2\tau))^{n/2}}2^n\pi^{n/2}.
\end{align*}
Where we have done the change of variables $u=\sqrt{(\coth\tau+ \tanh \tau)}z-\frac{\coth\tau-\tanh\tau}{\sqrt{(\coth\tau+ \tanh \tau)}}x$. This concludes the proof of $(1).$

By using the estimates of Remark \ref{'nota3'}, it is easy to show that
$$
\displaystyle |\partial_\tau e^{-\tau {\LL}}1(t,x) |\leq C\Big(\tanh(2\tau)+(1+\tanh^2(2\tau))|x|^2\Big)e^{-\tau {\LL}}1(t,x)
\leq C(\min\{\tau,1\}+|x|^2).$$

For  $(3)$,
consider first the case $|x|>1$. By the Mean Value Theorem  and parts  $(1), (2)$ in this Lemma we get
\begin{multline*} \Big| e^{-y \sqrt{{\LL}}} 1(t,x)-1 \Big| =
  \frac{1}{2\sqrt{\pi}}\bigg| \Big(\int_0^{1/|x|^2}+ \int_{1/|x|^2}^\infty\Big)\,  \frac{ye^{-\frac{y^2}{4\tau}}}{\tau^{1/2}}(e^{-\tau{\LL}}1(t,x)- 1)\frac{d\tau}{\tau}\bigg| \\ \leq C\int_0^{1/|x|^2}\frac{ye^{-\frac{y^2}{4\tau}}}{\tau^{1/2}}|x|^2\tau\frac{d\tau}{\tau}  +C\int_{1/|x|^2}^\infty\frac{ye^{-\frac{y^2}{4\tau}}}{\tau^{1/2}} \frac{d\tau}{\tau} \\
\underbrace{=}_{\substack{\frac{y^2}{4\tau}=v}} C\Big(|x|^2 y^2\int_{\frac{|x|^2y^2}{c}}^\infty v^{1/2}e^{-v}\frac{dv}{v^2} +\int_0^{\frac{|x|^2y^2}{4}} v^{1/2}e^{-{v}}\frac{dv}{v}\Big)
\\ 
\leq \frac{C |x|^2 y^2}{(|x|^2y^2)^{1-\alpha/2}}\int_{\frac{|x|^2y^2}{c}}^\infty v^{1/2-\alpha/2}e^{-v}\frac{dv}{v}+  C|x|^\alpha y^{\alpha}\int_0^{\frac{|x|^2y^2}{4}}v^{1/2-\alpha/2}e^{-{v}}\frac{dv}{v} \\ \leq  C \Gamma(1/2-\alpha/2) |x|^\alpha y^{\alpha}.
\end{multline*}
Regarding the case  $|x|<1$.
Again, by the Mean Value Theorem we get
\begin{multline*}
\Big|\int_0^{1}\frac{ye^{-\frac{y^2}{4\tau}}}{\tau^{1/2}}(e^{-\tau{\LL}}1(t,x)- 1) \frac{d\tau}{\tau}\Big|\\
\le C\int_0^1\frac{ye^{-\frac{y^2}{4\tau}}}{\tau^{1/2}}(\tau+|x|^2)\tau\frac{d\tau}{\tau}
\leq C\int_0^{1/|x|^2}\frac{ye^{-\frac{y^2}{4\tau}}}{\tau^{1/2}}|x|^2\tau\frac{d\tau}{\tau}+\int_0^1\frac{y e^{-\frac{y^2}{4\tau}}}{\tau^{1/2}}\tau^2 \frac{d\tau}{\tau}\\
\underbrace{\leq}_{\substack{\frac{y^2}{4\tau}=v}}  C|x|^2 y^2\int_{\frac{|x|^2y^2}{4}}^\infty v^{1/2}e^{-v}\frac{dv}{v^2}+C\int_{\frac{y^2}{4}}^\infty v^{1/2} e^{-v} \left(\frac{y^2}{v}\right)^2\frac{dv}{v}\\
\leq C |x|^\alpha y^{\alpha}\int_{\frac{|x|^2y^2}{4}}^\infty v^{1/2-\alpha/2}e^{-{v}}\frac{dv}{v}+C y^\alpha \int_{\frac{y^2}{4}}^\infty v^{1/2-\alpha/2} e^{-v} \frac{dv}{v}\leq  C \Gamma(1/2-\alpha/2) y^\alpha.
\end{multline*}
On the other hand,
\begin{align*}
\Big|\int_1^{\infty}\frac{ye^{-\frac{y^2}{4\tau}}}{\tau^{1/2}}(e^{-\tau{\LL}}1(t,x)- 1) \frac{d\tau}{\tau}\Big|
&\leq C\int_1^\infty\frac{y e^{-\frac{y^2}{4\tau}}}{\tau^{1/2}}\frac{d\tau}{\tau}\leq Cy^{\alpha}\int_1^\infty \frac{y^{1-\alpha}e^{-\frac{y^2}{4\tau}}}{\tau^{1/2-\alpha/2}}\frac{d \tau}{\tau}\\ &\le C \Gamma(1/2-\alpha/2) y^\alpha.
\end{align*}
 \end{proof}

\begin{lemma}\label{lemma5B}
	Let $	\mathcal{P}_y(\tau, x,z)$ the Poisson kernel  associated with the parabolic harmonic oscillator,  $\LL$, and given by $\eqref{solution}$. Then,
	\begin{itemize}
		\item [(i)] There exists a constant $C$ such that for every $x,z$ in $\R^{n}$ and $\tau>0$,  \newline $\displaystyle \Big |  	\mathcal{P}_y(\tau,x,z)\Big| \le Cy\, e^{-\frac{y^2+|z|^2}{c\tau}}\, \tau^{-(\frac{n+3}{2})}$, and      $\displaystyle \Big |  \partial_y^k \mathcal{P}_y(\tau,x,z)\Big| \le C_k\, e^{-\frac{y^2+|z|^2}{c\tau}}\tau^{-(\frac{n+k}{2}+1)}$, for $k\ge 1$.
		\item[(ii)]Let $\gamma,\nu\ge 0, s \ge 0$. For each  $\ell,k,m\in\N\cup\{0\}$, there exists a constant $C_{\gamma,\nu,k,\ell,m,s}>0$ such that, for every $x\in\R^n$  and $\tau>0$,
		$$
		\int_{\R^{n+1}}|x|^\gamma|z|^\nu|\partial^k_y \partial^m_{z_i}\partial^\ell_{x_j}	\mathcal{P}_y(\tau,x,z)|dzdz\leq\begin{cases}C_{\gamma,\nu,\ell,k,m,s}\, y^{-({k+m-\ell-\nu+\gamma+s})},\;
		& \mbox{if $s\ge0, \zeta>0$,}\\  C_{\gamma,\nu,\ell,k,m,s}\,y^{-s},\:\; 
		& \mbox{if $ s >0, \zeta\le 0$,} \end{cases}
		$$
		for  $\zeta =k+m-\ell-\nu+\gamma$ and  $i=1,\dots,n, \:\, j=1,\dots,n.$  
		\item[(iii)] Let $f$ such that $|x|^\alpha f\in L^\infty(\R^{n+1})$, $0<\alpha\le 1$ and $s\ge 0$. There exists a constant $C_{s,\alpha} >0$ such that, for every $x\in\R^n$ and $\tau>0$, $$\int_{\R^{n+1}}|\partial_{x_i}\partial_y^2	\mathcal{P}_y(\tau,x,z)f(t-\tau,x-z)|dzd\tau\le C_{s,\alpha} y^{-({1-\alpha+s})}. 
		$$
		\item [(iv)] There exists a constant $C$ such that for every $x\in \R^n$  and $\tau>0$,
		\begin{equation}\label{P4}
		\int_{\R^{n+1}}|\partial_{\tau}	\mathcal{P}_y(\tau,x,z) |dzd\tau\le Cy^{-2} .
		\end{equation}
	\end{itemize}
\end{lemma}

\begin{proof}
		Along this proof  will use  Remark \ref{'nota3'} and the estimates: \newline  $\displaystyle\partial_y^k \Big( \frac{ye^{-\frac{y^2}{4\tau}}  }{\tau^{3/2}} \Big) \le C_k \,e^{-\frac{y^2}{8\tau}}  \tau^{-(k/2+1)}$,  
$\hspace{0.2 cm}\displaystyle \Big|\partial_{x_i}^\ell \Big(  e^{-\frac{|2x-z|^2\tanh \tau}{4}} \Big)\Big| \le 	
C_\ell e^{-\frac{|2x-z|^2\tanh \tau}{8}} (\tanh \tau)^{\ell/2}$  and $\displaystyle |\partial_{z_i}^m e^{-\frac{|z|^2\coth\tau}{4}}|
\le C_m e^{-\frac{|z|^2\coth\tau}{8}} (\coth\tau)^{m/2}.$

Estimate (i) is consequence of Remark \ref{'nota3'}. In order to prove (ii), as
$$|\partial^k_y \partial^m_{z_i}\partial^\ell_{x_j}	\mathcal{P}_y(\tau,x,z)|\le 
\frac{C}{(\sinh \tau)^{n/2}} e^{-\frac{y^2}{C\tau}} \tau^{-(k/2+1)}  e^{-\frac{|z|^2 \coth \tau}{C}}(\coth\tau)^{m/2}e^{-{|x-\frac{z}{2}|^2}\tanh\tau}(\tanh \tau)^{\ell/2}  ,$$
again by  Remark \ref{'nota3'},  for every $\lambda \ge 0$ we get 
	\begin{multline*}
\int_{\R^{n+1}}|x|^\gamma|z|^\nu|\partial^k_y \partial^m_{z_i}\partial^\ell_{x_j}	\mathcal{P}_y(\tau,x,z)|dzdz   \\ \le  C\int_{\mathbb{R}^n}\int_0^\infty \frac{(|x-\frac{z}{2}|^\gamma +|\frac{z}{2}|^\gamma)|z|^\nu e^{-\frac{y^2}{C\tau}} e^{-\frac{|z|^2 \coth \tau}{C}e^{-{|x-\frac{z}{2}|^2}\tanh\tau}}} { \tau^{(k/2+1)} (\sinh \tau)^{n/2}}  (\coth\tau)^{m/2}(\tanh \tau)^{\ell/2} \frac{d\tau}{\tau}dz\\
	\le C\int_{\mathbb{R}^n}\int_0^\infty \frac{ e^{-\frac{y^2}{C\tau}} e^{-\frac{|z|^2 \coth \tau}{C}} }{ \tau^{\frac{k+n+m-\ell+\gamma-\nu+\lambda}{2}}}\frac{d\tau}{\tau} dz \le C\int_0^\infty \frac{ e^{-\frac{y^2}{C\tau}}  }{ \tau^{\frac{k+m-\ell+\gamma-\nu+\lambda}{2}}}\frac{d\tau}{\tau}.
	\end{multline*}
	The constant $C$ depends on  $\gamma,\nu,\ell,k,m$ and $\lambda.$
	The result follows by choosing $\lambda = s $ in the case  $k+m-\ell+\gamma-\nu>0$, for $k+m-\ell+\gamma-\nu \le 0$ we choose    $\lambda =-(k+m-\ell+\gamma-\nu)+s$  in the case $k+m-\ell+\gamma-\nu \le 0$.
	
	For (iii),  as $|x|^\alpha f\in L^\infty(\R^{n+1})$,  we have
	\begin{align*}
	&\int_{\R^{n+1}}|\partial_{x_i}\partial_y^2	\mathcal{P}_y(\tau,x,z)f(t-\tau,x-z)|dzd\tau\\&\le C \int_{\R^n}\int_0^\infty \frac{e^{-\frac{y^2}{c\tau}}e^{-\frac{|z|^2\coth\tau}{4}}e^{-\frac{|2x-z|^2\tanh\tau}{c}}\tanh\tau|2x-z|^{1-\alpha}|2x-z|^\alpha|f(t-\tau,x-z)|}{\tau (\sinh(2\tau))^{n/2}}\frac{d\tau}{\tau}dz\\
	&\le C \int_{\R^n}\int_0^\infty \frac{e^{-\frac{y^2}{c\tau}}e^{-\frac{|z|^2\coth\tau}{4}}(\tanh\tau)^{\frac{1+\alpha}{2}}
		|x-z|^\alpha|f(t-\tau,x-z)|}{\tau (\sinh(2\tau))^{n/2}}\frac{d\tau}{\tau}dz\\&+ C\int_{\R^n}\int_0^\infty \frac{e^{-\frac{y^2}{c\tau}}e^{-\frac{|z|^2\coth\tau}{4}}(\tanh\tau)^{\frac{1+\alpha}{2}}|z|^\alpha|f(t-\tau,x-z)|}{\tau (\sinh(2\tau))^{n/2}}\frac{d\tau}{\tau}dz\\
	&\le C[f]_{M^\alpha} \int_{\R^n}\int_0^\infty \frac{e^{-\frac{y^2}{c\tau}}e^{-\frac{|z|^2\coth\tau}{4}}
		(\tanh\tau)^{\frac{1+\alpha}{2}}}{\tau (\sinh(2\tau))^{n/2}}\frac{d\tau}{\tau}dz\\ & +C \|f\|_\infty\int_{\R^n}\int_0^\infty \frac{e^{-\frac{y^2}{c\tau}}e^{-\frac{|z|^2\coth\tau}{4}}\tanh\tau}{\tau (\sinh(2\tau))^{n/2}(\coth\tau)^{\alpha/2}}\frac{d\tau}{\tau}dz\\
	&\le C \int_0^\infty \frac{e^{-\frac{y^2}{c\tau}}}{\tau^{\frac{1-\alpha+\lambda}{2}}}\frac{d\tau}{\tau}.
	\end{align*}
The result follows by taking $\lambda=s .$
	
	We shall prove $(iv)$ in the case of the first derivative, we leave the details for the second derivative to the reader. By using the ideas in the proof of (iii) we have  
		\begin{multline*}
			\int_{\R^{n+1}}|\partial_{\tau} \mathcal{P}_y(\tau,x,z) |dzd\tau \\ \le 
	C\int_0^\infty\int_{\R^n} \frac{y e^{-\frac{y^2}{c\tau}}e^{-\frac{|z|^2\coth\tau}{c}}e^{-\frac{|2x-z|^2\tanh\tau}{c}}}{\tau^{3/2} (\sinh 2\tau)^{n/2} } \Big(  \frac1{\tau}+ \frac{\cosh(2\tau)}{\sinh (2\tau)} + \frac{|y|^2}{\tau^2} + \frac{|z|^2}{(\sinh \tau)^2} + \frac{|2x-z|^2}{(\cosh \tau)^2} \Big) d\tau \\
	\le C\int_0^\infty\int_{\R^n} \frac{y e^{-\frac{y^2}{c\tau}}e^{-\frac{|z|^2\coth\tau}{c}}e^{-\frac{|2x-z|^2\tanh\tau}{c}}}{\tau^{3/2} (\sinh 2\tau)^{n/2} } \Big(  \frac1{\tau}+ \frac{\cosh(2\tau)}{\sinh (2\tau)} + \frac{|y|^2}{\tau^2} + \frac{|z|^2}{(\sinh \tau)^2} + \frac{|2x-z|^2}{(\cosh \tau)^2} \Big) d\tau \\
	\le C\int_0^\infty\int_{\R^n} \frac{ e^{-\frac{y^2}{c\tau}}e^{-\frac{|z|^2}{c \tau}}}{ \tau^{1+n/2} }   \frac1{\tau}d\tau \le  C\int_0^\infty\frac{ e^{-\frac{y^2}{c\tau}}}{ \tau }   \frac{d\tau}{\tau} \le \frac{C}{y^2}.
\end{multline*}

\end{proof}

\end{document}